\documentclass[a4paper,11pt]{article}

\usepackage[small,bf,hang]{caption} 

\usepackage{todonotes}
\usepackage[ruled,linesnumbered]{algorithm2e}
\usepackage{amsmath,amsthm,amssymb}
\usepackage{graphicx,subcaption,tikz}
\usepackage{comment,xspace,paralist}							
\usepackage[shortlabels]{enumitem}	
\usepackage{hyperref}
\usepackage{changepage}
\hypersetup{
    colorlinks=true,
    citecolor = blue,
    linkcolor=blue,
    filecolor=magenta,
    urlcolor=cyan,
}
\usepackage[margin=1in]{geometry}
\usepackage{verbatim}
\usepackage[capitalize]{cleveref}
\usepackage{color}
\usepackage{authblk}

\usetikzlibrary{decorations.pathreplacing, calligraphy}

\newtheorem{theorem}{Theorem}[section]
\newtheorem{lemma}[theorem]{Lemma}

\newtheorem{claim}{Claim}[theorem]
\Crefname{claim}{Claim}{Claims}
\newtheorem{corollary}[theorem]{Corollary}
\Crefname{subsection}{Subsection}{Subsections}
\newtheorem{conjecture}{Conjecture}
\Crefname{conjecture}{Conjecture}{Conjectures}

\expandafter\let\expandafter\oldproof\csname\string\proof\endcsname
\let\oldendproof\endproof
\renewenvironment{proof}[1][\proofname]{%
	\oldproof[\normalfont\bfseries #1]%
}{\oldendproof}
\newenvironment{subproof}[1][\normalfont\it Subproof]{%
	\begin{proof}[#1]%
	}{%
	\end{proof}%
}

\title{Erd\H{o}s-Hajnal conjecture beyond five-vertex graphs}

	\author[1]{Shenwei Huang}
    
    \author[2]{Yiao Ju}

    \author[2]{Yidong Zhou}
	
	\affil[1] {School of Mathematical Sciences and LPMC, Nankai University, Tianjin 300071, P.R.~China. Email: {\tt shenweihuang@nankai.edu.cn}.}
  
    \affil[2] {College of Computer Science, Nankai University, Tianjin 300350, China.}

\begin{document}

\maketitle

\begin{abstract}
    In 1989, Erd\H{o}s and Hajnal conjectured that for any graph $H$, there is a constant $c=c(H)>0$ such that every $n$-vertex graph $G$ with no induced copies of $H$ contains a clique or an independent set of size at least $n^{c}$. This conjecture, known as the Erd\H{o}s-Hajnal conjecture, is a central open problem in combinatorics and listed as one of the top 10  Erd\H{o}s problems by Bloom on the Erd\H{o}s problem website https://www.erdosproblems.com/.
    In a recent breakthrough, Nguyen, Scott and Seymour proved that Erd\H{o}s-Hajnal conjecture holds for the case when $H$ is the five-vertex path, which, combined with known results, implies that Erd\H{o}s-Hajnal conjecture holds for every five-vertex graph.

    In this paper, we extend the iterative sparsification framework recently developed by Nguyen, Scott and Seymour. We introduce a generalized niceness condition relaxing their nice condition, a novel intermediate property concerning combs and a general structural lemma (which may be of independent interest)
    that is sufficient to deduce the Erd\H{o}s-Hajnal conjecture.

    This framework simultaneously recovers the recent result on the five-vertex path (PLMS 2026) and the classical result on the bull graph by Chudnovsky and Safra (JCTB 2008) as special cases, thereby unifying these two previously independent strands, and further proves the conjecture for two new cases: the E-graph (which contains the five-vertex path) and the Bird graph (which contains both the five-vertex path and the bull). These are the first two six-vertex graphs whose validity does not follow from the known operations (see Alon-Pach-Solymosi, Combinatorica 2001, and Nguyen-Scott-Seymour, TAMS 2026) that preserve the  Erd\H{o}s-Hajnal property.
\end{abstract}

\section{Introduction}

A {\em clique} ({\em stable set}) is a vertex set whose vertices are pairwise adjacent (nonadjacent, resp.).
Ramsey~\cite{Ra30} proved that for every integer $k\geq 1$, there exists a positive integer $R(k,k)$ such that every graph on at least $R(k,k)$ vertices contains a clique or stable set of size $k$. This shows that a sufficiently large graph must contain a large clique or stable set. Erd\H{o}s \cite{Er47} proved via probabilistic method that for large $n$, there exists an $n$-vertex graph with no cliques or stable sets of size at least $2\log n$. This shows that general graphs do not nacessarily contain a clique or stable set of polynomial size. However, $H$-free graphs might behave differently.

Let $G,H$ be graphs. We say that $G$ is {\em $H$-free} if $G$ has no induced subgraph that is isomorphic to $H$. 
In 1989, Erd\H{o}s and Hajnal \cite{EH89} showed that 
for any graph $H$, there exists a constant $c=c(H)>0$ such that 
any $n$-vertex $H$-free graph has a clique or stable set of size at least $e^{c\sqrt{\log n}}$.
Erd\H{o}s and Hajnal~\cite{EH89} furhter conjectured that one can improve the bound to a polynomial of $n$, which has become known as the Erd\H{o}s-Hajnal conjecture.

\begin{conjecture}\label{conj:E-H}
\textnormal{(Erd\H{o}s-Hajnal \cite{EH89})} For every graph $H$, there exists $c=c(H)>0$ such that every $n$-vertex $H$-free graph has a clique or stable set of size at least $n^c$.
\end{conjecture}

In 2024, Buci\'c, Nguyen, Scott and Seymour \cite{BNSS24} gave an $\log \log$ improvement over the classical Erd\H{o}s-Hajnal bound: for any graph $H$, there exists a constant $c=c(H)>0$ such that
any $n$-vertex $H$-free graph has a clique or stable set of size at least $e^{c\sqrt{\log n \log \log n}}$.
For a set $\mathcal{H}$ of graphs, $G$ is {\em $\mathcal{H}$-free} if $G$ is $H$-free for all $H\in\mathcal{H}$.
The Erd\H{o}s-Hajnal conjecture is equivalent to the statement that for any nonempty set $\mathcal{H}$ of graphs, there exists a constant $c=c(\mathcal{H})>0$ such that any $n$-vertex $\mathcal{H}$-free graph has a clique or stable set of size $n^c$.
If a nonempty set $\mathcal{H}$ of graphs satisfies this statement, we say that $\mathcal{H}$ has the {\em Erd\H{o}s-Hajnal property}. In this case, We call $c(\mathcal{H})$ a {\em Erd\H{o}s-Hajnal constant} for $\mathcal{H}$.
It is clear that a graph $H$ has the Erd\H{o}s-Hajnal property if and only if $\overline{H}$ has the Erd\H{o}s-Hajnal property, where $\overline{H}$ is the complement of $H$. 

\subsection{Equivalent formulations}

The R\"odl theorem says that every graph in a proper hereditary graph class contains an induced subgraph with linear size that is very sparse or very dense. Formally, we say that $X$ is {\em $c$-sparse} if every vertex in $X$ has at most $c|X|$ neighbors in $X$, and $X$ is {\em $c$-restricted} if $X$ or $\overline{X}$ is $c$-sparse. 

\begin{theorem}\label{thm:Rodl}
\textnormal{(R\"odl~\cite{Ro86})} For every graph $H$ and every $\epsilon\in(0,\frac{1}{2})$, there exists $\delta>0$ such that every $H$-free graph $G$ has an $\epsilon$-restricted induced subgraph of size at least $\delta|G|$.
\end{theorem}

The original proof of R\"odl uses the regularity lemma which gave
a tower-type bound on $\delta$.
Fox and Sudakov \cite{FS08} made a significant
improvement showing that one can take $\delta=e^{-c|H|(\log \frac{1}{\epsilon})^{2}}$ for some constant $c>0$, and they conjectured that one can take $\delta$ to be a polynomial in $\epsilon$. 
A set $\mathcal{H}$ of graphs (or the class of $\mathcal{H}$-free graphs) has the {\em polynomial R\"odl property}, if there exists $d\geq 1$ such that for every $\epsilon\in(0,\frac{1}{2})$, every $\mathcal{H}$-free graph $G$ has an $\epsilon$-restricted induced subgraph of size at least $\epsilon^d|G|$. By taking $\epsilon$ to be a small polynomial of $|G|^{-1}$, one may see that the polynomial R\"odl property implies the Erd\H{o}s-Hajnal property.

For two graphs $G$ and $H$, an induced {\em copy} of $H$ in $G$, is an injection $\phi:V(H)\rightarrow V(G)$ such that for all $u,v\in V(H)$, $uv\in E(H)$ if and only if $\phi(u)\phi(v)\in E(G)$. The number of induced copies of $H$ in $G$ is denoted by $\mathrm{ind}_H(G)$.
Nikiforov generalized Theorem \ref{thm:Rodl} to the graphs that contain few induced copies of $H$.

\begin{theorem}\label{thm:Nikiforov}
\textnormal{(Nikiforov~\cite{Ni06})} 
For every graph $H$ and every $\epsilon\in(0,\frac{1}{2})$, there exists $\delta>0$ such that every graph $G$ with $\mathrm{ind}_H(G)<(\delta|G|)^{|H|}$ has an $\epsilon$-restricted induced subgraph of size at least $\delta|G|$.
\end{theorem}

A finite set $\mathcal{H}$  of graphs (or the class of $\mathcal{H}$-free graphs) is {\em viral}, if there exists $d\geq 1$ such that every graph $G$ with $\mathrm{ind}_{H}(G)<(\epsilon^d|G|)^{|H|}$ for all $H\in\mathcal{H}$ has an $\epsilon$-restricted induced subgraph of size at least $\epsilon^d|G|$. Buci\'c, Fox and Pham~\cite{BFP24} proved the equivalence of the Erd\H{o}s-Hajnal property, the polynomial R\"odl property, and the viral property.

\begin{theorem}\label{thm:E-H-poly-Rodl-viral-equivalent}
\textnormal{(Buci\'c-Fox-Pham~\cite{BFP24})} 
For a finite set $\mathcal{H}$ of graphs, $\mathcal{H}$ has the Erd\H{o}s-Hajnal property if and only if it has the polynomial R\"odl property if and only if it is viral.
\end{theorem}

\subsection{Erd\H{o}s-Hajnal for a single graph}

A vertex set $S\subseteq V(G)$ is {\em homogeneous} in $G$ if all vertices in $S$ have the same neighborhood in $V(G)\setminus S$. A homogeneous set $S$ of $G$ is {\em nontrivial} if $2\leq |S|\leq |G|-1$. A graph $G$ is {\em prime} if $G$ has no nontrivial homogeneous set.

Let $G_1$ and $G_2$ be graphs with disjoint vertex sets, and $a\in V(G_1)$. We say that $G$ is obtained from $G_1$ by {\em substituting} $G_2$ for $a$ if:

$\bullet$ $V(G)=(V(G_1)\setminus \{a\})\cup V(G_2)$,

$\bullet$ $u,v\in V(G_1)\setminus \{a\}$ are adjacent in $G$ if and only if $u,v$ are adjacent in $G_1$,

$\bullet$ $u,v\in V(G_2)$ are adjacent in $G$ if and only if $u,v$ are adjacent in $G_2$,

$\bullet$ $u\in V(G_1)\setminus \{a\}$ and $v\in V(G_2)$ are adjacent in $G$ if and only if $u$ is adjacent to $a$ in $G_1$.

It is clear that a graph $G$ on at least two vertices is prime if and only if there do not exist graphs $G_1,G_2$ with $|G_1|,|G_2|\geq 2$ such that $G$ is obtained from $G_1$ by substituting $G_2$ for a vertex in $G_1$.

One important result due to Alon, Pach and Solymosi \cite{APS01} says that it suffices to consider Conjecture \ref{conj:E-H} for prime graphs.

\begin{theorem}\label{thm:E-H-substitute}
\textnormal{(Alon-Pach-Solymosi \cite{APS01})} Let $H_1,H_2$ be graphs, and $H$ be the graph obtained from $H_1$ by substituting $H_2$ for a vertex $v$ in $H_1$. If both $H_1,H_2$ have the Erd\H{o}s-Hajnal property, then $H$ also has the Erd\H{o}s-Hajnal property.
\end{theorem}

\begin{figure}[h!]
\centering
\begin{tikzpicture}[scale=0.6]
\tikzstyle{vertex}=[circle, draw, fill=white, inner sep=1pt, minimum size=5pt]

	\node[vertex](1) at (-1,1){};
	\node[vertex](2) at (-0.5,0){};
	\node[vertex](3) at (0.5,0){};
	\node[vertex](4) at (1,1){};
	\node[vertex](5) at (0,-1){};

	\foreach \from/\to in {1/2,2/3,3/4,2/5,3/5}
		\draw (\from) -- (\to);
\end{tikzpicture}
\caption{The bull graph.}
\label{fig:bull}
\end{figure}
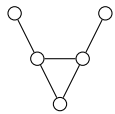

One can easily show via Theorem \ref{thm:E-H-substitute} that every graph with at most four vertices satisfies the Erd\H{o}s-Hajnal property.
There are four prime 5-vertex graphs: the bull (see Figure \ref{fig:bull}), $C_5$ (the five-vertex cycle), $P_5$ (the five-vertex path), and $\overline{P_5}$. 
In 2008, Chudnovsky and Safra \cite{CS08} showed that the bull graph has the Erd\H{o}s-Hajnal property. 
In 2023, Chudnovsky, Scott, Seymour and Spirkl \cite{CSSS23} showed that  
\begin{theorem}[\text{Chudnovsky-Scott-Seymour-Spirkl}  \cite{CSSS23}]\label{thm:C5 has EH property}
    $C_5$ has the Erd\H{o}s-Hajnal property. 
\end{theorem}

Also in 2023, Nguyen, Scott and Seymour \cite{NSS23} showed that 
\begin{theorem}[Nguyen-Scott-Seymour \cite{NSS23}]\label{thm:P5 has EH property}
    $P_5$ and $\overline{P_5}$ have the Erd\H{o}s-Hajnal property.  
\end{theorem}
These results imply that all graphs on at most 5 vertices have the Erd\H{o}s-Hajnal property. 
On the other hand, Nguyen, Scott and Seymour \cite{NSS232} constructed an infinitely family of prime graphs that have the Erd\H{o}s-Hajnal property. The main insight from \cite{NSS232} is an operation that preserves the viral property.

\begin{theorem}\label{thm:degree 1 and degree n-2}
\textnormal{(Nguyen-Scott-Seymour  \cite{NSS232})} 
    Let $\mathcal{F}$ be a finite set of graphs, and let $F_1,\overline{F_2}\in\mathcal{F}$. 
    For $i=1,2$, let $v_i$ be a vertex of $F_i$ with degree one, and let $F_i'=F_i\setminus\{v_i\}$. 
    If $\mathcal{F}_1=\{F_1'\}\cup(\mathcal{F}\setminus\{F_1\})$ and $\mathcal{F}_2=\{\overline{F_2’}\}\cup(\mathcal{F}\setminus\{\overline{F_2}\})$ are both viral, then $\mathcal{F}$ is viral.
\end{theorem}

By combining Theorems \ref{thm:E-H-poly-Rodl-viral-equivalent} and \ref{thm:degree 1 and degree n-2}, we have 
\begin{corollary}\label{coro:degree 1 and degree n-2}
    Let $\mathcal{F}$ be a finite set of graphs, and let $F_1,\overline{F_2}\in\mathcal{F}$. 
    For $i=1,2$, let $v_i$ be a vertex of $F_i$ with degree one, and let $F_i'=F_i\setminus\{v_i\}$. 
    If $\mathcal{F}_1=\{F_1'\}\cup(\mathcal{F}\setminus\{F_1\})$ and $\mathcal{F}_2=\{\overline{F_2’}\}\cup(\mathcal{F}\setminus\{\overline{F_2}\})$ have the Erd\H{o}s-Hajnal property, then $\mathcal{F}$ has the Erd\H{o}s-Hajnal property.
\end{corollary}

As a corollary, the two 6-vertex graphs in Figure \ref{fig:6-vertex-E-H}  both have the Erd\H{o}s-Hajnal property~\cite{NSS232}.

\begin{figure}[h!]
\centering
\begin{tikzpicture}[scale=1]
\tikzstyle{vertex}=[circle, draw, fill=white, inner sep=1pt, minimum size=5pt]
\tikzstyle{node}=[circle, draw, fill=black, inner sep=0.5pt, minimum size=1pt]

	\node[vertex](1) at (-1,1){};
	\node[vertex](2) at (0,1){};
	\node[vertex](3) at (1,1){};
	\node[vertex](4) at (-1,0){};
	\node[vertex](5) at (0,0){};
	\node[vertex](6) at (1,0){};

	\foreach \from/\to in {1/2,2/3,1/4,2/4,2/5,4/5,5/6}
		\draw (\from) -- (\to);
\end{tikzpicture}
\hspace{10mm}
\begin{tikzpicture}[scale=1]
\tikzstyle{vertex}=[circle, draw, fill=white, inner sep=1pt, minimum size=5pt]
\tikzstyle{node}=[circle, draw, fill=black, inner sep=0.5pt, minimum size=1pt]

	\node[vertex](1) at (-1,1){};
	\node[vertex](2) at (0,1){};
	\node[vertex](3) at (1,1){};
	\node[vertex](4) at (-1,0){};
	\node[vertex](5) at (0,0){};
	\node[vertex](6) at (1,0){};

	\foreach \from/\to in {1/2,2/3,1/4,2/4,2/5,3/5,4/5,5/6}
		\draw (\from) -- (\to);
\end{tikzpicture}
\caption{Two prime 6-vertex graphs that have the Erd\H{o}s-Hajnal property.}
\label{fig:6-vertex-E-H}
\end{figure}
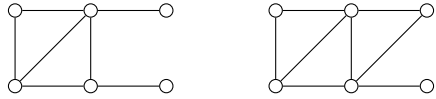

\subsection{Erd\H{o}s-Hajnal for two graphs}

In the following we mention some results on Erd\H{o}s-Hajnal property of graphs with more than one forbidden induced subgraph. 
Bousquet, Lagoutte and Thomass\'{e} \cite{BLT15} proved that for every integer $k\geq 1$, $\{P_k,\overline{P_k}\}$ has the Erd\H{o}s-Hajnal property. 
A stronger result that if $H_1,H_2$ are forests, then $\{H_1,\overline{H_2}\}$ has the Erd\H{o}s-Hajnal property~\cite{CSSS20} was proved by Chudnovsky, Scott, Seymour and Spirkl. In the paper which proved that $C_5$ has the Erd\H{o}s-Hajnal property~\cite{CSSS23}, the authors gave a generalized result. Let $H$ be a graph with vertex set $\{b_1,\ldots,b_k\}$. The {\em star-expansion} of $H$, is the graph obtained from $H$ that adds $k+1$ vertices $a_1,\ldots,a_k,v$, where $a_i$ is adjacent to $b_i$ for $i\in[k]$, and $v$ is adjacent to $a_1,\ldots,a_k$, and there are no other edges (see Figure \ref{fig:star expansion} for an example). They proved the following theorem.

\begin{figure}[h!]
\centering
\begin{tikzpicture}[scale=0.6]
\tikzstyle{vertex}=[circle, draw, fill=white, inner sep=1pt, minimum size=5pt]

	\node[vertex](1) at (-1.5,-0.5){};
	\node[vertex](2) at (-0.5,-0.5){};
	\node[vertex](3) at (0.5,-0.5){};
	\node[vertex](4) at (1.5,-0.5){};
	
	\node at (0,-2) {$P_4$};

	\foreach \from/\to in {1/2,2/3,3/4}
		\draw (\from) -- (\to);
\end{tikzpicture}
\hspace{10mm}
\begin{tikzpicture}[scale=0.6]
\tikzstyle{vertex}=[circle, draw, fill=white, inner sep=1pt, minimum size=5pt]

	\node[vertex](1) at (-1.5,-1){};
	\node[vertex](2) at (-0.5,-1){};
	\node[vertex](3) at (0.5,-1){};
	\node[vertex](4) at (1.5,-1){};

    \node[vertex](11) at (-1.5,0){};
	\node[vertex](21) at (-0.5,0){};
	\node[vertex](31) at (0.5,0){};
	\node[vertex](41) at (1.5,0){};

    \node[vertex](0) at (0,1){};

	\node at (0,-2) {star-expansion of $P_4$};

	\foreach \from/\to in {1/2,2/3,3/4}
		\draw (\from) -- (\to);

    \foreach \from/\to in {1/11,2/21,3/31,4/41}
		\draw (\from) -- (\to);

    \foreach \from/\to in {0/11,0/21,0/31,0/41}
		\draw (\from) -- (\to);
\end{tikzpicture}
\caption{$P_4$ and its star-expansion.}
\label{fig:star expansion}
\end{figure}
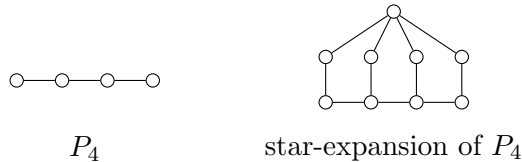

\begin{theorem}
\textnormal{(Chudnovsky-Scott-Seymour-Spirkl \cite{CSSS23})} 
Let $H$ be a forest. Let $H_1$ be the star expansion of $H$, and $H_2$ be the star-expansion of $\overline{H}$. Then $\{H_1,H_2,\overline{H_1},\overline{H_2}\}$ has the Erd\H{o}s-Hajnal property.
\end{theorem}

Since $P_4$ is the complement of itself and the star-expansion of $P_4$ contains $C_6,C_7$, it follows that $\{C_6,\overline{C_6}\}$ and $\{C_7,\overline{C_7}\}$ have the Erd\H{o}s-Hajnal property. In the same paper, it was proved that $\{\hat{C_5},\overline{\hat{C_5}}\}$ has the Erd\H{o}s-Hajnal property, where $\hat{C_5}$ is the graph obtained from $C_5$ by adding a vertex that is adjacent to two adjacent vertices of the $C_5$.

\subsection{Our Contributions}

\begin{figure}[h!]
\centering
\begin{tikzpicture}[scale=0.6]
\tikzstyle{vertex}=[circle, draw, fill=white, inner sep=1pt, minimum size=5pt]

	\node[vertex](1) at (0.5,1){};
	\node[vertex](2) at (-0.5,1){};
	\node[vertex](3) at (-0.5,0){};
	\node[vertex](4) at (-0.5,-1){};
	\node[vertex](5) at (0.5,-1){};
	\node[vertex](6) at (0.5,0){};
	\node at (0,-2) {$E-graph$};

	\foreach \from/\to in {1/2,2/3,3/4,4/5,3/6}
		\draw (\from) -- (\to);
\end{tikzpicture}
\hspace{10mm}
\begin{tikzpicture}[scale=0.6]
\tikzstyle{vertex}=[circle, draw, fill=white, inner sep=1pt, minimum size=5pt]

	\node[vertex](1) at (0,-0.3){};
	\node[vertex](2) at (-1,-1){};
	\node[vertex](3) at (-1,0.3){};
	\node[vertex](4) at (1,-1){};
	\node[vertex](5) at (1,0.3){};
	\node[vertex](6) at (0,1){};
	\node at (0,-2) {co-$E$};

	\foreach \from/\to in {1/2,1/3,1/4,1/5,2/3,2/4,3/5,4/5,3/6,5/6}
		\draw (\from) -- (\to);
\end{tikzpicture}
\caption{The graphs $E$-graph and its complement co-$E$.}
\label{fig:E}
\end{figure}
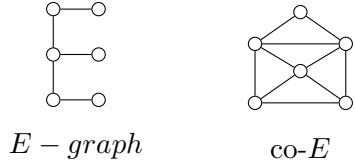

\begin{figure}[h!]
\centering
\begin{tikzpicture}[scale=0.6]
\tikzstyle{vertex}=[circle, draw, fill=white, inner sep=1pt, minimum size=5pt]

	\node[vertex](1) at (0.7,0){};
	\node[vertex](2) at (-0.7,0){};
	\node[vertex](3) at (0,-1){};
	\node[vertex](4) at (1,1){};
	\node[vertex](5) at (2,1){};
	\node[vertex](6) at (-2,0){};
	\node at (0,-2) {Bird};

	\foreach \from/\to in {1/2,2/3,3/1,1/4,4/5,2/6}
		\draw (\from) -- (\to);
\end{tikzpicture}
\hspace{10mm}
\begin{tikzpicture}[scale=0.6]
\tikzstyle{vertex}=[circle, draw, fill=white, inner sep=1pt, minimum size=5pt]

	\node[vertex](1) at (0,-0.3){};
	\node[vertex](2) at (-1,-1){};
	\node[vertex](3) at (-1,0.3){};
	\node[vertex](4) at (1,-1){};
	\node[vertex](5) at (1,0.3){};
	\node[vertex](6) at (0,1){};
	\node at (0,-2) {co-Bird};

	\foreach \from/\to in {1/2,1/3,1/5,2/3,2/4,3/5,4/5,3/6,5/6}
		\draw (\from) -- (\to);
\end{tikzpicture}
\caption{The graph Bird and its complement co-Bird.}
\label{fig:Bird}
\end{figure}
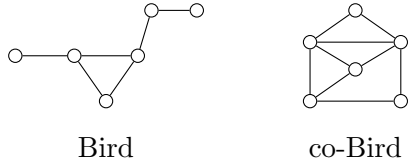

    Let $E$-graph and Bird be the graphs in Figures \ref{fig:E} and \ref{fig:Bird}, respectively. 
    In this paper, we prove the following theorems. 
    \begin{theorem}\label{thm:E}
        $E$-graph has the Erd\H{o}s-Hajnal property. 
    \end{theorem}

    \begin{theorem}\label{thm:Bird}
        Bird has the Erd\H{o}s-Hajnal property. 
    \end{theorem}
    
    Since both $E$-graph and Bird contain an induced $P_5$, both Theorem \ref{thm:E} and Theorem \ref{thm:Bird} generalize the main result by Nguyen, Scott and Seymour \cite{NSS23}. Since Bird graph contains the bull graph, Theorem \ref{thm:Bird} generalizes the result of Chudnovsky and Safra \cite{CS08}. In fact, as we shall see, these two previously independent results are unified within our extended iterative sparsification framework (see the remark at the end of Section \ref{sec:proving (*)}).

    While the extension from five-vertex graphs to six-vertex graphs may appear incremental, the technical obstacles are substantial: the existing framework of Nguyen, Scott and Seymour does not directly apply to these cases, and new structural tools-a generalized niceness condition and a comb-based criterion-are required.
    
    Our proofs are inspired by \cite{NSS23}. 
    The proof of Theorems \ref{thm:E} and \ref{thm:Bird} proceeds in four main steps. First, 
    we extend the framework of Nguyen, Scott and Seymour \cite{NSS23} by introducing a generalized niceness condition that relaxes the original notion of "nice", and show that any finite class of graphs satisfying this property, along with two technical conditions called leaf-reducibility and wonderfulness (see Section \ref{sec:pre} for precise definitions), has the Erd\H{o}s-Hajnal property (see Section \ref{sec:iterative-sparsification-2}).
    
    \begin{lemma}\label{lem:reduce EH to gene-nice}
        Let $\mathcal{F}$ be a finite class of graphs that is leaf-reducible and wonderful. If $\mathcal{F}$ is generalized nice, then $\mathcal{F}$ has the Erd\H{o}s-Hajnal property. 
    \end{lemma}
    
    Second, we introduce a new structural property, denoted by  $(*)$, which concerns the existence of certain structures from a comb, and prove that under leaf-reducibility condition, $(*)$ implies generalized niceness (see Section \ref{sec:iterative-sparsification-1}).

    \begin{lemma}\label{lem:reduce gene-nice to comb-property}
        Let $\mathcal{F}$ be a finite class of graphs that is leaf-reducible.
        If $\mathcal{F}$ has property $(*)$, then
        $\mathcal{F}$ is generalized nice. 
    \end{lemma}

    Third, we establish a general structural lemma that provides sufficient conditions for deducing $(*)$ (see Section \ref{sec:proving (*)}).

    \begin{lemma}\label{lem:proving (*)}
       
        Let $\mathcal{F}_1$ and $\mathcal{F}_2$ be two finite sets of graphs that satisfy the Erd\H{o}s-Hajnal property. 
        Let $\mathcal{H}$ be a finite set of graphs.
        Suppose that for every $\overline{\mathcal{H}}$-free graph $G$ and every $(\ell,w)$-comb $((a_i,B_i),i\in[\ell])$ in $G$ with $\ell,w\geq 4$, if there exists a special vertex $v\in V(G)\setminus(\{a_i:i\in[\ell]\}\cup\bigcup_{i\in[\ell]}B_i)$ such that $v$ is complete to $\bigcup_{i\in[\ell]}B_i$ and anti-complete to $\{a_i:i\in[\ell]\}$,
        $B_i$ (for each $i\in[\ell]$) can be partitioned into $X_i,Y_i$ such that
        \begin{itemize}
            \item[$(1)$] $Y_i$ is $\mathcal{F}_1$-free;  
            \item[$(2)$] $X_i$ can be partitioned into $(A^i_1,\ldots, A^i_{t_i})$ such that 
            \begin{itemize}
                \item[$(2.1)$] $(A^i_1,\ldots, A^i_{t_i})$ is a pure blockade; 
                \item[$(2.2)$] the pattern of $(A^i_1,\ldots, A^i_{t_i})$, whose vertex set consists of all blocks of $(A^i_1,\ldots, A^i_{t_i})$ and two vertices are adjacent if and only if their corresponding blocks are complete to each other, is $\mathcal{F}_2$-free; 
                \item[$(2.3)$] for each $j\in [t_i]$ and every vertex $u\in \bigcup_{k\in [\ell]\setminus \{i\}} B_k$, $u$ is pure to $A^i_j$.
            \end{itemize}
        \end{itemize} 
        Then $\mathcal{H}$ satisfies property $(*)$.
    \end{lemma}

    We make essential use of the comb structure, which was introduced by Chudnovsky, Scott, Seymour and Spirkl \cite{CSSS23}. While the relevance of combs to the Erdős-Hajnal conjecture has been noted previously, to the best of our knowledge, this is the first time that a systematic framework-encompassing property $(*)$, generalized niceness, and two new graph classes-has been built around them. In particular, our Lemma \ref{lem:proving (*)} provides a general structural criterion that goes beyond any individual application of combs in the literature.
    
    Finally, we show that $E$-graph and Bird satisfy the hypothesis of Lemma \ref{lem:proving (*)}
    (see Section \ref{sec:deducing structure}).  
    One key step in the proof is to define certain equivalence relation on the set of blocks in a blockade and take quotient to obtain a new blockade whose pattern graph lies in a graph class which can be proved to have the Erd\H{o}s-Hajnal property. 

\section{Preliminaries}\label{sec:pre}

All graphs in this paper are finite and simple. We follow~\cite{BM08} for general graph theory terminology that is not defined here. Let $P_n,C_n$ and $K_n$ denote the induced path, cycle and complete graph one $n$ vertices, respectively.
For positive integers $s,t\ge 1$, let $K_{s,t}$ be the complete bipartite graph with one side having $s$ vertices and other side having $t$ vertices. 
If $s=1$, the graph $K_{1,t}$ is called a {\em star}. 
For a positive integer $k$, let $[k]$ denote the set $\{1,2,\ldots,k\}$. 
For a set of graphs $\mathcal{H}$, let $\overline{\mathcal{H}}=\{\overline{H}:H\in \mathcal{H}\}$. 
For a graph $G$, we say $G$ is {\em connected} ({\em anti-connected}) if for every two vertices $v_1,v_2\in G$, there is a path in $G$ ($\overline{G}$, resp.) with ends $v_1,v_2$. 
We say $S$ is a {\em connected component} ({\em anti-connected component}) of $G$, if $S$ is a maximal connected (anti-connected, resp.) induced subgraph of $G$. 
Let $G$ be a graph, and let $X,Y\subseteq V(G)$ be disjoint. 
Let $E(X,Y)$ be the set of edges with one end in $X$ and one end in $Y$. 
We say that $X$ is {\em complete} ({\em anti-complete}) to $Y$, if every vertex in $X$ is adjacent (non--adjacent, resp.) to every vertex in $Y$. 
We say that $X$ is {\em pure} to $Y$ if $X$ is complete or anti-complete to $Y$, and $X$ and $Y$ are {\em mixed} otherwise. If $X=\{x\}$, we say that $x$ is mixed on $Y$.

For $c\in(0,1)$, $X$ is {\em $c$-sparse} to $Y$ if every vertex in $X$ has at most $c|Y|$ neighbors in $Y$, and $X$ and $Y$ are {\em weakly $c$-sparse} if $|E(X,Y)|\leq c|X||Y|$. 
For convenience, we say that $X$ and $Y$ are $c$-sparse if $X$ is $c$-sparse to $Y$ and $Y$ is $c$-sparse to $X$.
An {\em $(\ell,w)$-blockade} in a graph $G$, is a sequence $\mathcal{B}=(B_1,\ldots,B_k)$ of disjoint subsets of $V(G)$, where $k\geq \ell$ and $|B_i|\geq w$ for all $i\in[k]$. 
Each $B_i$ in  $\mathcal{B}$ is called a {\em block}.
Let $V(\mathcal{B})=\bigcup_{i=1}^{k} B_i$.
A blockade $\mathcal{B}=(B_1,\ldots,B_k)$ is {\em complete} ({\em anti-complete}, {\em pure}, {\em weakly $x$-sparse}) if its blocks are pairwise complete (anti-complete, pure, weakly $x$-sparse, resp.). 
We say $\mathcal{B}=(B_1,\ldots,B_k)$ is {\em $x$-sparse} if $B_i$ is $x$-sparse to $B_j$ for all $i,j\in[k]$ with $i>j$.

For a positive integer $\ell$ and $w\geq0$, an {\em $(\ell,w)$-comb} in a graph $G$, is a sequence of pairs $((a_i,B_i),i\in[\ell])$ such that the following holds:

$\bullet$ $(B_1,\ldots,B_\ell)$ is an $(\ell,w)$-blockade;

$\bullet$ $a_1,\ldots,a_\ell$ are distinct;

$\bullet$ $\{a_1,\ldots,a_\ell\}, B_1,\ldots,B_\ell$ are disjoint subsets of $V(G)$;

$\bullet$ for all distinct $i,j\in[\ell]$, $a_i$ is complete to $B_i$ and anti-complete to $B_j$.

\subsection{Leaf-reducibility and wonderfulness}

For a finite class of graphs $\mathcal{F}$, we say that $\mathcal{F}$ is {\em leaf-reducible} if there exists some graph $H\in \mathcal{F}$ such that $H$ has a vertex $v$ of degree 1 and $(\mathcal{F}\setminus \{H\})\cup \{(H-v)\}$ satisfies the Erd\H{o}s-Hajnal property.
By Theorem \ref{thm:E-H-substitute} and the fact that $P_4$ has the Erd\H{o}s-Hajnal property, $E$-graph and Bird are leaf-reducible.
We say that $\mathcal{F}$ is {\em wonderful} if there is a constant $a\geq 6$ such that the following holds for every $y\in (0,1/2)$ and every $\overline{\mathcal{F}}$-free graph $G$.
If $\mathcal{B}=(B_1,\ldots,B_\ell)$ is a $(\ell,w)$-blockade where $\ell$ is an integer at least $y^{-a}$, each block of $\mathcal{B}$ is anti-connected and has the same size, and every two blocks in $\mathcal{B}$ are complete or $y^{a}$-sparse, then one of the following holds.

\begin{itemize}
    \item $G$ has a $y^{4}$-restricted induced subgraph of size at least $w$. 
    \item There exists $i\in[\ell]$, such that there are at most $y|G|$ vertices $v\in V(G)\setminus V(\mathcal{B})$ such that $|N(v)\cap B_i|\in(0,\frac{1}{2}|B_i|)$.
\end{itemize}

For a graph $H$ with two special vertices $v_1$ and $v_2$ in $H$, let $H^+$ be the graph obtained from $H$ by adding a new vertex $v$, two edges $vv_1,vv_2$ and an edge $v_1v_2$ if $v_1v_2\notin E(H)$, and $H^-$ be the graph obtained from $H$ by adding a vertex $v$ and two edges $vv_1,vv_2$ and deleting an edge $v_1v_2$ if $v_1v_2\in E(H)$. We prove that $E$-graph and Bird are wonderful.
    
\begin{lemma}\label{lem:wonderful}
    If $\mathcal{F}$ satisfies one of the following, then $\mathcal{F}$ is wonderful. 
    \begin{itemize}
       
        \item[$(1)$] $\mathcal{F}$ contains an induced subgraph of 1-subdivision of $K_{1,t}$ for some $t\geq 1$. 
        
        \item[$(2)$] There is a graph $H$ with two distinct special vertices $v_1,v_2\in V(H)$ such that $\{H\}\cup \overline{\mathcal{F}}$ has the Erd\H{o}s-Hajnal property and neither $H^+$ nor $H^-$ is $\overline{\mathcal{F}}$-free.
           
    \end{itemize}
        
\end{lemma}

\begin{proof}   
    We shall determine $a$ later.
    Let $y\in (0,1/2)$ and  $G$ be $\overline{\mathcal{F}}$-free.
    Let $\mathcal{B}=(B_1,\ldots,B_\ell)$ be a $(\ell,w)$-blockade where $\ell$ is an integer at least $y^{-a}$, each block of $\mathcal{B}$ is anti-connected and has the same size, and every two blocks are complete or $y^{a}$-sparse.
    
    Let $I(v)$ be the set of indices $i\in[\ell]$ such that $|N(v)\cap B_i|\in(0,\frac{1}{2}|B_i|)$.
    If $|I(v)|\leq y\ell$ for every vertex $v\in V(G)\setminus V(\mathcal{B})$, then  
    the second bullet of wonderfulness holds by double counting. 
    So we may suppose that there exists $v\in V(G)\setminus V(\mathcal{B})$ with $|I(v)|\geq y\ell$. 
    Without loss of generality, we may assume that $I(v)=[|I(v)|]$. 
    Let $J$ be the graph with $V(J)=I(v)$ and $ij\in E(J)$ if and only if $B_i$ is complete to $B_j$. Since $|N(v)\cap B_i|\in (0,\frac{1}{2}|B_i|)$ and $B_i$ is anti-connected for every $i\in I(v)$, there exists a non-edge $b_ib_i'\in B_i$ such that $v$ is adjacent to $b_i$ but is non-adjacent to $b_i'$.

    \begin{claim}\label{clm:pattern-EH}
        Let $a_0=1$ if $(1)$ holds and $a_0$ be the maximum size of a graph in $\{H\}\cup \overline{\mathcal{F}}$ if $(2)$ holds. 
        Then there exists a constant $c\in (0,1)$ such that if $a\geq a_0$, then $J$ has a clique or stable set of size at least $|J|^c$.
    \end{claim}
    
    \begin{subproof}[Proof of Claim \ref{clm:pattern-EH}]
        Suppose first that (1) holds, i.e., $\mathcal{F}$ contains an induced subgraph of 1-subdivision of $K_{1,t}$ for some $t\geq 1$.
        We claim that $c=t^{-1}$ suffices. 
        If $J$ has a clique of size $t$, say induced by $[t]$, then $\{v,b_1,b_1',\ldots,b_t,b_t'\}$ induces the complement of 1-subdivision of $K_{1,t}$. 
        Therefore, $J$ has no clique of size $t$ and so has a stable set of size at least $|J|^{t^{-1}}$ by the Ramsey theorem \cite{ES35}. 

        Now suppose that (2) holds. 
        Since $\{H\}\cup \overline{\mathcal{F}}$ has the Erd\H{o}s-Hajnal property, there is a constant $c>0$ such that every $(\{H\}\cup \overline{\mathcal{F}})$-free graph of size $n$ has a clique or a stable set of size at least $n^{c}$.
        It suffices to prove that $J$ is $(\{H\}\cup \overline{\mathcal{F}})$-free. 

        Suppose by contradiction that $J$ has an induced copy of some graph $F$ in $\{H\}\cup \overline{\mathcal{F}}$. 
        For convenience, we may assume that $V(F)=[s]$.  
        Suppose first that $F\in \overline{\mathcal{F}}$.
        Let $v_1$ be a vertex in $B_1$ and
        we define $v_2,\ldots,v_s$ inductively as follows. 
        Suppose that $v_1,\ldots,v_{j-1}$ have been defined with $2\le j\leq s$. 
        Since $s\leq a_0\leq a$ and $a\geq 1$, we have $(s-1)y^a\leq (a-1)y^a\leq (a-1)(\frac{1}{2})^a<\frac{1}{2}$. 
        It follows that $1-\frac{1}{2}-(s-1)y^a>0$, which implies that one can choose a vertex $v_j\in B_j$ such that $v_iv_j\in E(G)$ if and only if $ij\in E(J)$ for all $i\in[j-1]$.
        So $\{v_1,\ldots,v_s\}$ induces an $F$ in $G$, a contradiction. 
        Now suppose that $F=H$. 
        Without loss of generality, we may assume that $1,2$ are the special vertices of $H$. 
        Since $v$ is mixed on $B_i$ for every $i\in [s]$, we may take $v_1\in B_1,v_2\in B_2$ such that $vv_1,vv_2\in E(G)$. 
        We define $v_3,\ldots,v_s$ inductively. 
        Suppose that $v_1,\ldots,v_{j-1}$ are defined where $3\le j\leq s$. 
        Since $1-\frac{1}{2}-(s-1)y^a>0$, we could take a vertex $v_j\in B_j\setminus N(v)$ such that $v_iv_j\in E(G)$ if and only if $ij\in E(J)$ for all $i\in [j-1]$. 
        Then $\{v,v_1,\ldots,v_s\}$ induces an $H^+$ or $H^-$, depending on whether $v_1v_2\in E(G)$. This contradicts that neither $H^+$ nor $H^-$ is $\overline{\mathcal{F}}$-free.
    \end{subproof}

    Let $a=\max\{a_0,\frac{5}{c}+1\}\geq 6$ where $a_0,c$ are given by Claim \ref{clm:pattern-EH}. 
    We claim that $a$ satisfies Lemma \ref{lem:wonderful}. 
    Since $a\geq a_0$, $J$ has a clique or a stable set $R$ of size $r=\lceil|J|^c\rceil$ by Claim \ref{clm:pattern-EH}.
    By taking complement if necessary, we may assume that $R$ be a stable set. Let $S=\bigcup_{i\in R}B_i$. Since each block of $\mathcal{B}$ has the same size, each vertex $v\in B_i\subseteq S$ has at most $r^{-1}|S|$ neighbors in $B_i$ and at most $y^a|S|$ neighbors in $S\setminus B_i$. Since 
    \begin{align*}
        r^{-1}+y^a & \leq (y\ell)^{-c}+y^a  & (r\ge y\ell)\\
                   &  \leq y^{c(a-1)}+y^a   & (\ell \geq y^{-a}) \\
                   &  \leq y^5+y^5          & (a\ge \frac{5}{c}+1\ge 6) \\
                   &   \leq y^4,            & (y<1/2)
    \end{align*}
    $S$ is $y^4$-sparse. This gives the first outcome of Lemma \ref{lem:wonderful}.
\end{proof}

\begin{lemma}\label{lem:E and Bird are wonderful}
    $E$-graph and Bird are wonderful. 
\end{lemma}
\begin{proof}
    By Lemma \ref{lem:wonderful}, $E$-graph is wonderful since $E$-graph is an induced subgraph of 1-subdivision of $K_{1,3}$. 
    Let $H$ be the graph labeled as in Figure \ref{fig:H}. 
    Since $\{v_1,v_3\}$ is a homogeneous set of $H$, $H$  has the Erd\H{o}s–Hajnal property by Theorem \ref{thm:E-H-substitute}. 
    Let $H'$ be the graph obtained from $H$ by adding a new vertex $v$ and edges $vv_1$ and $vv_2$. It follows that $H'-v_5$ is a co-bird and $(H'-v_1v_2)-v_6$ is a co-Bird, which implies that Bird is wonderful by Lemma \ref{lem:wonderful}.
\end{proof}
\begin{figure}[h!]
        \centering
        \begin{tikzpicture}[scale=1]
        \tikzstyle{vertex}=[circle, draw, fill=white, inner sep=1pt, minimum size=5pt]
        
            \node[vertex](1) at (45+19:1.5){\scriptsize $v_1$};
            
            \node[vertex](3) at (45-19:1.5){\scriptsize $v_3$};

            \node[vertex](2) at (135:1.5){\scriptsize $v_2$};
            \node[vertex](4) at (135+90:1.5){\scriptsize $v_4$};
            \node[vertex](6) at (135+90+90:1.5){\scriptsize $v_6$};

            \node[vertex](5) at (0,0){\scriptsize $v_5$};

            \foreach \from/\to in {1/2,2/4,4/6,6/1,3/1,3/2,3/6}
            \draw (\from) -- (\to);

            \foreach \from/\to in {5/1,5/2,5/3,5/4,5/6}
            \draw (\from) -- (\to);

            \foreach \from/\to in {6/1,6/4}
            \draw (\from) -- (\to);

            \foreach \from/\to in {1/3}
            \draw (\from) -- (\to);

        \end{tikzpicture}
        \caption{A graph $H$ with the Erd\H{o}s-Hajnal property.}
        \label{fig:H}
    \end{figure}
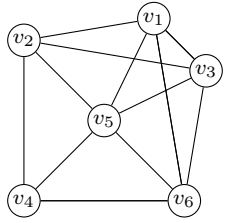

\subsection{Lemmas from literature}

We will use several results about restricted induced subgraphs and blockades in a graph. 
The following lemmas are proved in \cite{NSS23}. 

\begin{lemma}\label{lem:complete-blockade-or-large-anticonnected-component}
\textnormal{(Lemma 4.1 of Nguyen-Scott-Seymour \cite{NSS23})} 
Let $k\geq 2$ be an integer. If $G$ is a graph whose anti-connected components have size less than $|G|/k$, then there is a complete $(k,|G|/k^2)$-blockade in $G$.
\end{lemma}

\begin{lemma}\label{lem:give a poly-rodl subgraph}
\textnormal{(Theorem 7.4 of Nguyen-Scott-Seymour \cite{NSS23})}
    Let $\epsilon\in(0,\frac{1}{2})$ and $a\geq 1$. If $G$ is a graph such that for every induced subgraph $F$ of $G$ with $|F|\geq \epsilon^{2a}|G|$, there exists a complete or anti-complete $(k,|F|/k^a)$-blockade in $F$ with $k\in [2,\epsilon^{-1}]$, then
    $G$ has an $\epsilon$-restricted induced subgraph with at least $\epsilon^{3a}|G|$ vertices.
\end{lemma}

\begin{lemma}\label{lem:pure or sparse to pairwise complete or weakly sparse}
\textnormal{(Theorem 6.1 of Nguyen-Scott-Seymour \cite{NSS23})}
    Let $\epsilon\in(0,\frac{1}{2})$, $d\geq 1$, and $x=\epsilon^{5d}$. 
    If $G$ is a graph with $|G|\geq \epsilon^{-10d^2}$ such that
    for every induced subgraph $F$ of $G$ with $|F|\geq \epsilon^{d}|G|$, there exists a pure or $x$-sparse $(k,|F|/k^d)$-blockade in $F$ with $k\in [2,x^{-1}]$, then there is an $(\epsilon^{-1},x^{2d}|G|)$-blockade $(B_1,\ldots,B_{\ell})$ in $G$ such that for all distinct $i,j\in[\ell]$, $(B_i,B_j)$ is either complete or weakly $\epsilon^d$-sparse in $G$.
\end{lemma}

The following lemmas are implicit in \cite{NSS232,NSS25,NSS23}

\begin{lemma}\label{lem:weakly sparse to sparse}
\textnormal{(contained in Lemma 7.1 of Nguyen-Scott-Seymour \cite{NSS23})}
    Let $a\geq 0$, $0<\epsilon\leq \frac{1}{4}$ and $m>0$. If a graph $G$ has a blockade $(A_1,\ldots,A_{\ell})$ with $\ell=\lceil \epsilon^{-1}\rceil$ and $|A_i|\geq m$,  
    then there is a blockade $(D_1,\ldots,D_\ell)$ such that 
    \begin{itemize}
        \item $D_i$ is a subset of $A_i$ with $|D_i|=\lceil\epsilon \lceil m\rceil\rceil$ for every $i\in [\ell]$,  
        \item if $A_i$ and $A_j$ are weakly $\epsilon^{a}$-sparse, then $D_i$ is $\epsilon^{a-5}$-sparse to $D_j$ and $D_j$ is $\epsilon^{a-5}$-sparse to $D_i$, and
        \item if $A_i$ and $A_j$ are weakly $(1-\epsilon^{a})$-dense, then $D_i$ is $(1-\epsilon^{a-5})$-dense to $D_j$ and $D_j$ is $(1-\epsilon^{a-5})$-dense to $D_i$.
    \end{itemize}
\end{lemma}

\begin{lemma}\label{lem:leaf expanding}
\textnormal{(contained in Lemma 5.1 of Nguyen-Scott-Seymour \cite{NSS232})}
    If a finite set $\mathcal{F}$ of graphs is leaf-reducible, then there exist constants $d>0$ and $h\ge 1$ such that for every $y>0,b>1$ and every $y$-sparse $\mathcal{F}$-free graph $G$, one of the following holds.
    \begin{itemize}
        \item There are disjoint $X,Y\subseteq V(G)$ with $|X|\geq y^{bd+1}|G|$ and $|Y|\geq (1-hy)|G|$ such that $Y$ is anti-complete to $X$.
        
        \item $G$ has a $y^b$-restricted induced subgraph of size at least $y^{bd+1}|G|$.
    \end{itemize}
\end{lemma}

\begin{lemma}\label{lem:get a blockade by iteration}
\textnormal{(contained in Lemma 5.3 of Nguyen-Scott-Seymour \cite{NSS23})}
Let $x\in(0,\frac{1}{2})$, $a>1, b>0$, $c=2^{-4b}$ and $y\in(0,c]$. 
If $G$ is a graph with $|G|\geq y^{-(a+2)}$ such that for every induced subgraph $F$ of $G$ with $|F|\geq c|G|$, there are disjoint $X,Y\subseteq V(F)$ such that $|X|\geq y^a|F|$, $|Y|\geq (1-by)|F|$, and $Y$ is $x$-sparse or complete to $X$, then there is an $x$-sparse or complete $(y^{-1},y^{a+2}|G|)$-blockade in $G$.
\end{lemma}

\begin{proof}[Proof of Lemma \ref{lem:get a blockade by iteration}]
    Let $n$ be a maximal integer such that $G$ has a blockade $(B_1,\ldots,B_n)$ such that $|B_{i}|\geq y^{a+2}|G|$ for all $i\in[n]$ and $|B_n|\geq (1-by)^n|G|$, and for all $i\in [n]$, either $B_j$ is $x$-sparse to $B_i$ for all $j\in[n]$ with $j>i$ or $B_j$ is complete to $B_i$ for all $j\in[n]$ with $j>i$. 
    This is possible since we may take $n=1$ and $B_1=V(G)$. 
    Suppose that $n< 2y^{-1}$. 
    Note that $\log (1-t)\geq -2t\log2$ for every $t\in (0,\frac{1}{2})$. 
    Since $y\leq c=2^{-4b}$, $by\leq b2^{-4b}<\frac{1}{2}$. 
    This implies that $(1-by)^{2y^{-1}}=\exp (2y^{-1}\log(1-by))\geq \exp (2y^{-1}(-2by\log 2))=\exp(-4b\log 2)=2^{-4b}=c$. 
    Since $|B_n|\geq (1-by)^n|G|\geq (1-by)^{2y^{-1}}|G|\geq 2^{-4b}|G|=c|G|$, there are disjoint $X,Y\subseteq B_n$ such that $|X|\geq y^a|B_n|\geq y^a\cdot c|G|\geq y^{a+2}|G|$, $|Y|\geq (1-by)|B_n|\geq (1-by)^{n+1}|G|$, and $Y$ is $x$-sparse or complete to $X$, and for all $i\in[n-1]$, $X\cup Y$ is $x$-sparse (complete) to $B_i$ if $B_n$ is $x$-sparse (complete, resp.) to $B_i$. 
    So $(B_1,\ldots,B_{n-1},X,Y)$ contradicts the maximality of $n$.
    This proves that $n\geq 2y^{-1}$.
    Let $Q,R$ be subsets of $[n]$ such that $Q$ is the set of indices $i$ such that $B_j$ is $x$-sparse to $B_i$ for all $j\in [n]$ with $j>i$ and $R$ is the set of indices $i$ such that $B_j$ is complete to $B_i$ for all $j\in [n]$ with $j>i$.
    By the definition of $(B_1,\ldots,B_n)$, each index $i\in[n]$ is in one of $Q,R$. 
    So one of $Q,R$ has size at least $\frac{n}{2}\geq y^{-1}$.
    Therefore, one of $(B_i:i\in Q)$ and $(B_i:i\in R)$ is a $x$-sparse or complete $(y^{-1},y^{a+2}|G|)$-blockade.
\end{proof}

It is easy to see that Lemma \ref{lem:get a blockade by iteration} still holds if we replace the word ``$x$-sparse"  with ``anti-complete".

\begin{lemma}\label{lem:give a more sparse subgraph 2}
\textnormal{(contained in Lemma 3.2 of Nguyen-Scott-Seymour \cite{NSS25})} 
    Let $c\in (0,1)$, $b_1>1, b_2,b_3>0$, and $b_1b_2\geq b_2+b_3$. 
    Suppose that $x\in (0,c)$ and $G$ is a graph satisfying: 
    \begin{itemize}
        \item there is a $c$-sparse induced subgraph of $G$ with at least $c^{b_2}|G|$ vertices, and 
        \item for every $\lambda\in [x,c]$ and every $\lambda$-sparse induced subgraph $F$ of $G$ with $|F|\geq \lambda^{b_2}|G|$, there is a $\lambda^{b_1}$-sparse induced subgraph of $F$ with at least $\lambda^{b_3}|F|$ vertices. 
    \end{itemize}
    Then $G$ contains an $x$-sparse induced subgraph with at least $x^{b_1b_2}|G|$ vertices. 
\end{lemma}

It is easy to see that Lemma \ref{lem:give a more sparse subgraph 2}
still holds if we replace ``sparse'' throughout by ``restricted''.

\begin{lemma}\label{lem:get a comb in a sparse graph}
\textnormal{(contained in Lemma 5.2 of Nguyen-Scott-Seymour \cite{NSS23})}
Let $0<x\leq y\leq 2^{-8}$. For every $y^3$-sparse graph $G$ with $|G|\geq y^{-4}$, one of the following holds.

\begin{itemize}
    \item There are disjoint $X,Y\subseteq V(G)$ such that $|X|\geq y^4|G|$, $|Y|\geq (1-4y)|G|$, and $Y$ is $x$-sparse to $X$.
    
    \item $G$ is $2y^4$-sparse.
    
    \item For some integer $\ell\in[y^{-1},x^{-2}]$, there is an $(\ell, y^4|G|/\ell^2)$-comb $((a_i,B_i):i\in[\ell])$ in $G$, and a vertex $v\in V(G)\setminus(\{a_i:i\in[\ell]\}\cup\bigcup_{i\in[\ell]}B_i)$ such that $v$ is complete to $\bigcup_{i\in[\ell]}B_i$ and anti-complete to $\{a_i:i\in[\ell]\}$.
\end{itemize}

\end{lemma}

\section{Reducing the Erd\H{o}s-Hajnal property to the generalized niceness}\label{sec:iterative-sparsification-2}

    In this section, we prove Lemma \ref{lem:reduce EH to gene-nice} which says $\mathcal{F}$ has the Erd\H{o}s-Hajnal property if $\mathcal{F}$ is generalized nice and $\mathcal{F}$ is leaf-reducible and wonderful. 
    The method used in this section and next section is originated from \cite{NSS23,NSS25}.
    Here, we apply and adapt this framework to a setting that goes beyond the cases treated in their work.

    A finite set of graphs $\mathcal{F}$ is {\em generalized nice} if there exist $c_1\geq 3$, $c_2\geq 8$, $c_3,c_4,c_5,c_8>0$, and $c_6\geq 1$, $c_7\geq 4$ such that for every $\overline{\mathcal{F}}$-free graph $G$ and every $0<\epsilon<\frac{1}{2}$, either  
        \begin{itemize}
            \item $G$ has an $(\epsilon^{-1},  \epsilon^{c_1}|G|)$-blockade whose blocks are pairwise complete or weakly $\epsilon^{c_2}$-sparse; 

            \item $G$ has a clique or stable set of size at least $(\epsilon^{c_3}|G|)^{c_4}$;

            \item $G$ has a complete or anti-complete $(k,|G|/k^{c_5})$-blockade with $k\geq \epsilon^{-c_6}$; 

            \item $G$ has an $\epsilon^{c_7}$-restricted induced subgraph $S$ of size at least $\epsilon^{c_8}|G|$.  
        \end{itemize}

    We now translates the first outcome of generalized niceness to a long complete blockade, a restricted induced subgraph or an anti-complete pair. 
    
    \begin{lemma}\label{lem:main-E-graph-structure}
        Suppose that $\mathcal{F}$ is generalized nice, leaf-reducible and wonderful. 
        Then there exist constants $a_1,a_2,a_5>0$ and $a_3\geq a_4\geq 4$ such that for every $0<y<\frac{1}{2}$ and every $y$-restricted $\overline{\mathcal{F}}$-free graph $G$, one of the following holds.
        \begin{itemize}
            \item[$(1)$] $G$ has a clique or stable set of size at least $(y^{a_1}|G|)^{a_2}$. 

            \item[$(2)$] $G$ has a $y^{a_4}$-restricted induced subgraph $S$ with $|S|\geq y^{a_3}|G|$.

            \item[$(3)$] $G$ has a complete or anti-complete $(k,\frac{|G|}{k^{a_3}})$-blockade with $k\geq y^{-1}$.

            \item[$(4)$] There are disjoint $X,Y\subseteq V(G)$ such that $|X|\geq y^{a_3}|G|$, $|Y|\geq (1-a_5y)|G|$, and $Y$ is anti-complete or complete to $X$. 
            
        \end{itemize}
    \end{lemma}
    
    \begin{proof}   
        Let $c_1\geq 3$, $c_2\geq 8$, $c_3,c_4,c_5,c_8>0$, and $c_6\geq 1$, $c_7\geq 4$ be given by the generalized niceness of $\mathcal{F}$. 
        Let $d,h>0$ be given by Lemma \ref{lem:leaf expanding} and $a\geq 6$ be given by wonderfulness of $\mathcal{F}$.  
        We claim that $a_1=ac_3$, $a_2=c_4$, $a_3=a(c_1+c_8+5)+c_5+4d+1$, $a_4=4$ and $a_5=4+h$ suffice. 

        If $|G|\leq y^{-a_3}$, then the second outcome of Lemma \ref{lem:main-E-graph-structure} holds. 
        So we may assume that $|G|\geq y^{-a_3}$. Suppose first that $\overline{G}$ is $y$-sparse. 
        Since $\mathcal{F}$ is leaf-reducible,
        it follows from Lemma \ref{lem:leaf expanding} with $b=4$ that one of the following holds.
        \begin{itemize}
            \item There are disjoint $X,Y\subseteq V(G)$ with $|X|\geq y^{4d+1}|G|$, $|Y|\geq (1-hy)|G|$, such that $Y$ is complete to $X$.
            
            \item $G$ has a $y^4$-restricted induced subgraph of size at least $y^{4d+1}|G|$.
        \end{itemize}
        They give the second or the fourth outcome of Lemma \ref{lem:main-E-graph-structure}. 
        So we may assume that $G$ is $y$-sparse. 
        Let $\epsilon=y^a$. 
        By the generalized niceness of $\mathcal{F}$, one of the following holds.  
        \begin{itemize}
            \item $G$ has an $(\epsilon^{-1},  \epsilon^{c_1}|G|)$-blockade whose blocks are pairwise complete or weakly $\epsilon^{c_2}$-sparse.

            \item $G$ has a clique or stable set of size at least $(\epsilon^{c_3}|G|)^{c_4}$.

            \item $G$ has a complete or anti-complete $(k,|G|/k^{c_5})$-blockade with $k\geq \epsilon^{-c_6}$.

            \item $G$ has an $\epsilon^{c_7}$-restricted induced subgraph $S$ of size at least $\epsilon^{c_8}|G|$.  
        \end{itemize}
        
        If the second bullet holds, then the first outcome of Lemma \ref{lem:main-E-graph-structure} holds since $(\epsilon^{c_3}|G|)^{c_4}\geq (y^{a_1}|G|)^{a_2}$. If the third bullet holds, then the third outcome of Lemma \ref{lem:main-E-graph-structure} holds since $\frac{|G|}{k^{c_5}}\geq  \frac{|G|}{k^{a_3}}$ and $\epsilon^{-c_6}\geq y^{-1}$. 
        If the fourth bullet holds, then the second outcome of Lemma \ref{lem:main-E-graph-structure} holds since $\epsilon^{c_7}=y^{ac_7}\leq y^{a_4}$ (as $c_7\ge 4$) and $\epsilon^{c_8}=y^{ac_8}\geq y^{a_3}$.

        So we may assume that the first bullet holds, that is, $G$ has a blockade $(A_1,\ldots,A_\ell)$ with $\ell=\lceil \epsilon^{-1}\rceil$, and each block has size at least $m= \epsilon^{c_1}|G|\geq \epsilon^{c_1}\cdot y^{-a_3}\geq \epsilon^{c_1}\cdot \epsilon^{-(c_1+5)}= \epsilon^{-5}>1$ such that the blocks are pairwise complete or weakly $\epsilon^{c_2}$-sparse.
        
        Let $J$ be the graph with vertices set $[\ell]$, and $ij\in E(J)$ if and only if $A_i$ is complete to $A_j$. 
        Note that $\epsilon=y^a\leq 2^{-6}$. 
        By Lemma \ref{lem:weakly sparse to sparse}, there is a blockade $(D_1,\ldots,D_\ell)$ with $D_i\subseteq A_i$ such that $|D_i|=\lceil\epsilon \lceil m\rceil\rceil\geq \epsilon m$ for every $i\in [\ell]$ and $D_i$ is $\epsilon^{c_2-5}$-sparse to $D_j$ for every $ij\in E(\overline{J})$.
        
        Suppose that for some $i\in[\ell]$, $D_i$ has no anti-connected components of size at least $\frac{|D_i|}{\ell}$. 
        By Lemma \ref{lem:complete-blockade-or-large-anticonnected-component}, $D_i$ contains a complete $(\ell,\frac{|D_i|}{\ell^2})$-blockade. 
        Since $\ell\geq \epsilon^{-1}= y^{-a}\geq y^{-1}$, $\frac{|D_i|}{\ell^2}=\frac{\lceil\epsilon \lceil m\rceil\rceil}{\lceil \epsilon^{-1}\rceil^2}\geq \frac{\epsilon\cdot \epsilon^{c_1}|G|}{\epsilon^{-4}}= \epsilon^{c_1+5}|G|=y^{a(c_1+5)}|G|\geq y^{a_3}|G|$. 
        This gives the third outcome of Lemma \ref{lem:main-E-graph-structure}, a contradiction. 
        So for each $i\in[\ell]$, $D_i$ has an anti-connected component of size at least $\frac{|D_i|}{\ell}$. 
        By anti-connectivity, we may take an anti-connected subset $B_i\subseteq D_i$ with $|B_i|=\lceil \frac{|D_i|}{\ell} \rceil$ for all $i\in [\ell]$. Note that
        $|B_i|\ge \frac{|D_i|}{\ell}\geq \epsilon^2 |D_i|\ge \epsilon ^3 m\ge \epsilon^{-2}>1$. 
        Since $D_i$ is $\epsilon^{c_2-5}$-sparse to $D_j$ and $|B_j|\geq \epsilon^2 |D_j|$, $B_i$ is $\epsilon^{c_2-7}$-sparse to $B_j$ for all $ij\in E(\overline{J})$.

        Let $B=V(G)\setminus(\bigcup_{i\in[\ell]}B_i)$. 
        Since $\mathcal{F}$ is wonderful and $\epsilon^{c_2-7}= y^{a(c_2-7)}\leq y^a$ (as $c_2\ge 8$), one of the following holds. 
        \begin{itemize}
            
            \item $G$ has a $y^{4}$-restricted induced subgraph of size at least $\epsilon^3m$.  
            
            \item There exists $i\in[\ell]$, such that there are at most $y|G|$ vertices $v\in V(G)\setminus V(\mathcal{B})$ with $|N(v)\cap B_i|\in(0,\frac{1}{2}|B_i|)$.
            
        \end{itemize}
        If $G$ has a $y^{4}$-restricted induced subgraph of size at least $\epsilon^3m\geq y^{a(c_1+3)}|G|$, then the second outcome of Lemma \ref{lem:main-E-graph-structure} holds. 
        So we may assume that there exists $i\in[\ell]$ such that there are at most $y|G|$ vertices $v\in V(G)\setminus V(\mathcal{B})$ such that $|N(v)\cap B_i|\in(0,\frac{1}{2}|B_i|)$. 
        
        Since $G$ is $y$-sparse, there are at most $2y|G|$ vertices in $B$ that has at least $\frac{|B_i|}{2}$ neighbors in $B_i$.  
        Since $m>1$, $|D_i|\le \lceil m \rceil\le 2m$. 
        Since $\frac{|D_i|}{\ell}>1$, $|B_i|\le 2\frac{|D_i|}{\ell}$. This implies that 
        \begin{align*}
            \sum_{i=1}^{\ell}|B_i| & \le \ell\cdot \frac{4m}{\ell} =4m\\
            &  \leq \epsilon^{-2}\cdot \epsilon^{c_1}|G| & (0<\epsilon<1/2)\\
            &\le \epsilon |G| & (c_1\ge 3)\\
            & \le y|G|.          & (\epsilon \le y)  
        \end{align*}
        Let $Y$ be the set of vertices in $D$ that is anti-complete to $D_i$.
        Then $|Y|\geq |G|-\sum_{i=1}^{\ell}|B_i|-y|G|-2y|G|\geq (1-4y)|G|$.
        Since $|B_i|\geq \epsilon^{3}m= \epsilon^{c_1+3}|G|=y^{a(c_1+3)}|G|\geq y^{a_3}|G|$, $B_i$ and $Y$ satisfy the fourth outcome of Lemma \ref{lem:main-E-graph-structure}. 
        
        This complete the proof of Lemma \ref{lem:main-E-graph-structure}. 
    \end{proof}
   
    We then present a lemma to translate the fourth outcome of Lemma \ref{lem:main-E-graph-structure} to an anti-complete or complete blockade. 

    \begin{lemma}\label{lem:main-once-a-block-turn2}
        Suppose that $\mathcal{F}$ is generalized nice, leaf-reducible and wonderful. Then
        there exist constants $c\in (0,\frac{1}{2}),c_1,c_2>0$ and $c_3\geq c_4\geq 4$ such that for $y\in (0,c]$ and every $cy$-restricted $\overline{\mathcal{F}}$-free graph $G$, one of the following holds.
        \begin{itemize}
            \item[$(1)$] $G$ has a clique or stable set of size at least $(y^{c_1}|G|)^{c_2}$.
            
            \item[$(2)$] $G$ has a complete or anti-complete $(k,\frac{|G|}{k^{c_3}})$-blockade with $k\geq y^{-1}$.
            
            \item[$(3)$] $G$ has a $y^{c_4}$-restricted induced subgraph $S$ with $|S|\geq y^{c_3}|G|$.
            
        \end{itemize}
    \end{lemma}
    
    \begin{proof}
        Let $a_1,a_2,a_5>0$ and $a_3\geq a_4\geq 4$ be given by Lemma \ref{lem:main-E-graph-structure}. 
        We claim that $c=2^{-4a_5},c_1=a_1+1$, $c_2=a_2$, $c_3=a_3+2$ and $c_4=a_4$ suffice.
        If $|G|\leq y^{-c_3}$, then the third outcome of Lemma \ref{lem:main-once-a-block-turn2} holds. 
        So we may assume that $|G|\geq y^{-c_3}$. 
        
        Suppose  first that for every induced subgraph $F$ of $G$ with $|F|\geq c|G|$, there are disjoint $X,Y\subseteq V(F)$ such that $|X|\geq y^{a_3}|F|$, $|Y|\geq (1-a_5y)|F|$, and $Y$ is anti-complete or complete to $X$.
        By Lemma \ref{lem:get a blockade by iteration} with $a=a_3$, $b=a_5$, and ``$x$-sparse" replaced by ``anti-complete", there is an anti-complete or complete $(y^{-1},y^{a_3+2}|G|)$-blockade in $G$, which gives the second outcome of Lemma \ref{lem:main-once-a-block-turn2}.
        
        Therefore, we may assume that there is an induced subgraph $F$ of $G$ with $|F|\geq c|G|$ such that there is no disjoint $X,Y\subseteq V(F)$ such that $|X|\geq y^{a_3}|F|$, $|Y|\geq (1-a_5y)|F|$, and $Y$ is anti-complete or complete to $X$.
        
        Since $G$ is $cy$-restricted, $F$ is $y$-restricted. 
        By Lemma \ref{lem:main-E-graph-structure} with $G$ replaced by $F$, one of the following holds. 
        \begin{itemize}
            \item $F$ has a clique or stable set of size at least $(y^{a_1}|F|)^{a_2}$.

            \item $F$ has a $y^{a_4}$-restricted induced subgraph $S$ with $|S|\geq y^{a_3}|F|$.

            \item $F$ has a complete or anti-complete $(k,\frac{|F|}{k^{a_3}})$-blockade with $k\geq y^{-1}$. 
        \end{itemize}
        
        If the first bullet holds, then the first outcome of Lemma \ref{lem:main-once-a-block-turn2} holds since $(y^{a_1}|F|)^{a_2}\geq (y^{a_1+1}|G|)^{a_2}=(y^{c_1}|G|)^{c_2}$. 
        If the second bullet holds, then the third outcome of Lemma \ref{lem:main-once-a-block-turn2} holds since $y^{a_3}|F|\geq y^{a_3+1}|G|\geq y^{c_3}|G|$ and $a_4=c_4$. 
        If the third bullet holds, then the second outcome of Lemma \ref{lem:main-once-a-block-turn2} holds since $\frac{|F|}{k^{a_3}}\geq \frac{|G|}{y^{-1}\cdot k^{a_3}}\geq \frac{|G|}{k^{a_3+1}}\geq \frac{|G|}{k^{c_3}}$.
        
        This complete the proof of Lemma \ref{lem:main-once-a-block-turn2}
    \end{proof}
    We then use the method of iterative sparsification to decrease the sparse condition of Lemma \ref{lem:main-once-a-block-turn2} to a constant sparse condition. 
    \begin{lemma}\label{lem:main-iterative-sparsification-turn2}
        Suppose that $\mathcal{F}$ is generalized nice, leaf-reducible and wonderful. Then
        there exist constants $c\in (0,\frac{1}{2}), a_1\geq 1,a_2>0$ such that for every $x\in (0,c^2)$ and every $c^2$-restricted $\overline{\mathcal{F}}$-free graph $G$, one of the following holds.
        \begin{itemize}
            \item[$(1)$] $G$ has an $x$-restricted induced subgraph $S$ with $|S|\geq x^{a_1}|G|$.
            
            \item[$(2)$] $G$ has a clique or stable set of size at least $(x^{a_1}|G|)^{a_2}$.
            
            \item[$(3)$] $G$ has a complete or anti-complete $(k,\frac{|G|}{k^{a_1}})$-blockade with $k\geq 2$.
        \end{itemize}
    \end{lemma}
    
    \begin{proof}
        Let $c\in (0,\frac{1}{2}),c_1,c_2>0$ and $c_3\geq c_4\geq 4$ be given by Lemma \ref{lem:main-once-a-block-turn2}. 
        We claim that $c, a_1=c_1+3c_3$ and $a_2=c_2$ suffice. 
        Suppose that none of outcomes of Lemma \ref{lem:main-iterative-sparsification-turn2} holds. 
        We first present a claim to give the condition of Lemma         \ref{lem:give a more sparse subgraph 2}. 
        
        \begin{claim}\label{clm:condition of lemma-give a more sparse subgraph}
            For every $y$ with $cy\in[x,c^2]$ and every $cy$-restricted induced subgraph $F$ of $G$ with $|F|\geq (cy)^{\frac{4c_3}{c_4}}|G|$, there is a $(cy)^{\frac{c_4}{2}}$-restricted induced subgraph of $F$ with at least $(cy)^{c_3}|F|$ vertices. 
        \end{claim}
        \begin{subproof}[Proof of Claim \ref{clm:condition of lemma-give a more sparse subgraph}]
            By Lemma \ref{lem:main-once-a-block-turn2} with $G$ replaced by $F$, one of the following holds. 
            \begin{itemize}
            \item $F$ has a clique or stable set of size at least $(y^{c_1}|F|)^{c_2}$.
            
            \item $F$ has a complete or anti-complete $(k,\frac{|F|}{k^{c_3}})$-blockade with $k\geq y^{-1}$.

            \item $F$ has a $y^{c_4}$-sparse induced subgraph $S$ with $|S|\geq y^{c_3}|F|$. 
         
        \end{itemize}
        
        If the first bullet holds, then the second outcome of Lemma \ref{lem:main-iterative-sparsification-turn2} holds since $(y^{c_1}|F|)^{c_2}
        \geq (y^{c_1}\cdot (cy)^{\frac{4c_3}{c_4}}|G|)^{c_2}
        \geq (y^{c_1}\cdot (y^2)^{\frac{4c_3}{c_4}}|G|)^{c_2}
        = (y^{c_1+\frac{8c_3}{c_4}}|G|)^{c_2}
        \geq (x^{c_1+\frac{8c_3}{c_4}}|G|)^{c_2}
        \geq (x^{a_1}|G|)^{a_2}$, where the last inequality is due to $c_4\ge 4$. 

        If the second bullets holds, then the third outcome of Lemma \ref{lem:main-iterative-sparsification-turn2} holds since $\frac{|F|}{k^{c_3}}\geq \frac{y^{\frac{8c_3}{c_4}}|G|}{k^{c_3}}\geq \frac{|G|}{k^{c_3+\frac{8c_3}{c_4}}}\geq \frac{|G|}{k^{a_1}}$.

        So we may assume that the third outcome holds. Since  
        $y^{c_4}=(y^2)^{\frac{c_4}{2}}\leq (cy)^{\frac{c_4}{2}}$
        and $y^{c_3}\ge (cy)^{c_3}$,
        Claim \ref{clm:condition of lemma-give a more sparse subgraph} follows. 
        \end{subproof}

        Let $b_1=\frac{c_4}{2},b_2=\frac{4c_3}{c_4},b_3=c_3$. 
        Since $c_4\geq 4$, $b_1b_2=2c_3\geq c_3+\frac{4c_3}{c_4}=b_2+b_3$.   
        By Claim \ref{clm:condition of lemma-give a more sparse subgraph} and Lemma \ref{lem:give a more sparse subgraph 2} (with $b_1,b_2,b_3$ and $\lambda$ replaced by $cy$, $c$ replaced by $c^2$, sparse replaced by restricted), $G$ has an $x$-restricted induced subgraph with at least $x^{2c_3}|G|$ vertices. 
        Since $x^{2c_3}\geq x^{a_1}$, the first outcome of Lemma \ref{lem:main-iterative-sparsification-turn2} holds. 

        This completes the proof of Lemma \ref{lem:main-iterative-sparsification-turn2}. 
    \end{proof}
    We then present a lemma to obtain the polynomial R\"odl property.  
 
    \begin{lemma}\label{lem:main-rodl-turn2}
        Suppose that $\mathcal{F}$ is generalized nice, leaf-reducible and wonderful. Then
        there exist constants $c_1\geq 1,c_2>0$ such that for every $x\in(0,\frac{1}{2})$ and for every $\overline{\mathcal{F}}$-free graph $G$, one of the following holds. 
        \begin{itemize}
            \item[$(1)$] $G$ has an $x$-restricted induced subgraph $S$ with $|S|\geq x^{c_1}|G|$. 
            
            \item[$(2)$] $G$ has a complete or anti-complete $(k,|G|/k^{c_1})$-blockade with $k\geq 2$.
            
            \item[$(3)$] $G$ has a clique or stable set of size at least $(x^{c_1}|G|)^{c_2}$.
        \end{itemize}
    \end{lemma}
    
    \begin{proof}
        Let $c\in (0,\frac{1}{2}),a_1\geq1$ and $a_2>0$ be given by Lemma \ref{lem:main-iterative-sparsification-turn2}. 
        Let $\xi=c^2$. 
        By Theorem \ref{thm:Rodl} (with $\epsilon=\xi$), there exists $\delta>0$ such that every $\overline{\mathcal{F}}$-free graph $G$ has a $\xi$-restricted induced subgraph of size at least $\delta|G|$. 
        Let $c_2=a_2$ and $c_1$ be sufficiently large such that $\delta\geq \xi^{\frac{c_1}{2}}$, $\delta\geq 2^{a_1-c_1}$ and $c_1\geq 2a_1$. 
        In the following we prove that $c_1,c_2$ suffice. 
        
        By Theorem \ref{thm:Rodl}, $G$ has a $\xi$-restricted induced subgraph $F$ with $|F|\geq \delta|G|$. 
        If $x\geq \xi$, then $F$ is $x$-restricted with $|F|\geq \delta|G|\geq 2^{a_1-c_1}|G|\geq (x^{-1})^{a_1-c_1}|G|= x^{c_1-a_1}|G|\geq x^{c_1}|G|$, which gives the first outcome of Lemma \ref{lem:main-rodl-turn2}. 
        So we may assume that $x<\xi=c^2$. 
        By Lemma \ref{lem:main-iterative-sparsification-turn2} with $G$ replaced by $F$, one of the following holds. 
        \begin{itemize}
            \item $F$ has an $x$-restricted induced subgraph $S$ with $|S|\geq x^{a_1}|F|$. 
            
            \item $F$ has a clique or stable set of size at least $(x^{a_1}|F|)^{a_2}$.
            
            \item $F$ has a complete or anti-complete $(k,\frac{|F|}{k^{a_1}})$-blockade with $k\geq 2$. 
        \end{itemize}
        Note that 
        \begin{align*}
        x^{a_1}|F| & \geq x^{a_1}\delta|G| & (|F|\geq \delta|G|) \\
            & \geq x^{a_1}\xi^{\frac{c_1}{2}}|G| & (\delta\geq \xi^{\frac{c_1}{2}})\\
            & \geq x^{a_1+\frac{c_1}{2}}|G| & (\xi> x)\\
            & \geq x^{c_1}|G|. & (2a_1\leq c_1) 
        \end{align*}
        
        If the first bullet holds, then since $|S|\geq x^{a_1}|F|\geq x^{c_1}|G|$, the first outcome of Lemma \ref{lem:main-rodl-turn2} holds. 
        If the second bullet holds, then since $G$ has a clique or stable set of size at least $(x^{a_1}|F|)^{a_2}\geq (x^{c_1}|G|)^{c_2}$, the third outcome of Lemma \ref{lem:main-rodl-turn2} holds. 
        So we may assume that the third bullet holds.
        Since $\delta\geq 2^{a_1-c_1}$ and $k\geq 2$, we have $\frac{|F|}{k^{a_1}}\geq \frac{\delta|G|}{k^{a_1}}\geq \frac{2^{a_1-c_1}|G|}{k^{a_1}}\geq \frac{|G|}{k^{c_1}}$. 
        Therefore, the second outcome of Lemma \ref{lem:main-rodl-turn2} holds. 
        This complete the proof of Lemma \ref{lem:main-rodl-turn2}. 
    \end{proof}
    
    Now we are ready to prove Lemma \ref{lem:reduce EH to gene-nice}, which we restate here for readers' convenience. 

    \begin{lemma}\label{lem:reduce EH to gene-nice 2}
        Let $\mathcal{F}$ be a finite class of graphs that is leaf-reducible and wonderful. If $\mathcal{F}$ is generalized nice, then $\mathcal{F}$ has the Erd\H{o}s-Hajnal property. 
    \end{lemma}
    
    \begin{proof}
        Let $c_1\geq 1,c_2>0$ be given by Lemma \ref{lem:main-rodl-turn2}. 
        Let $q=42c_1^2$ and $m=2^q$, $c=\min\{q^{-1},\frac{c_2}{2}\}$. 
        We claim that every $\overline{\mathcal{F}}$-free graph $G$ has a clique or stable set of size at least $|G|^c$. 
        Suppose to the contrary that $G$ has no clique or stable set of size at least $|G|^c$. 
        Since every graph on two or more vertices has a clique or stable set of size 2, we may assume that $|G|>m$. 
     
        Let $$x=|G|^{-\frac{1}{3c_1}}\leq m^{-\frac{1}{3c_1}}=2^{-\frac{42c_1^2}{3c_1}}=2^{-14c_1}$$ 
        and 
        $$\epsilon=x^{\frac{1}{7c_1}}\leq 2^{-\frac{14c_1}{7c_1}}=\frac{1}{4}.$$ 
        Since $c_1\geq 1$, we have $x<\epsilon\leq\frac{1}{4}$. 
        We then present a claim to give some property of an induced subgraph of $G$ with at least $\epsilon^{2c_1}|G|$ vertices. 
        \begin{claim}\label{clm:a large F of G has a complete or anti-complete blockade with small length}
            Every induced subgraph $F$ of $G$ with $|F|\geq \epsilon^{2c_1}|G|$ has a complete or anti-complete $(k,|F|/k^{c_1})$-blockade with $k\in[2,\epsilon^{-1}]$.
        \end{claim}
        \begin{subproof}[Proof of Claim \ref{clm:a large F of G has a complete or anti-complete blockade with small length}]
            Suppose to the contrary that there is an induced subgraph $F$ of $G$ with $|F|\geq \epsilon^{2c_1}|G|$ such that $F$ has no complete or anti-complete $(k,|F|/k^{c_1})$-blockades with $k\in[2,\epsilon^{-1}]$.
            By Lemma \ref{lem:main-rodl-turn2} with $G$ replaced by $F$, one of the following holds.  
            \begin{itemize}
                \item $G$ has an $x$-restricted induced subgraph $S$ with $|S|\geq x^{c_1}|F|$.
            
                \item $G$ has a complete or anti-complete $(k,|F|/k^{c_1})$-blockade with $k\geq \epsilon^{-1}$.
            
                \item $G$ has a clique or stable set of size at least $(x^{c_1}|F|)^{c_2}$. 
            \end{itemize}
            
            Suppose first that the first bullet holds. 
            Since $x=|G|^{-\frac{1}{3c_1}}$ and $a_1\geq 1$, $|S|\geq x^{c_1}|F|\geq x^{c_1+\frac{2}{7}}|G|=x^{c_1+\frac{2}{7}}\cdot x^{-3c_1}\geq x^{-1}$. 
            By taking complement if necessary, we may assume that $S$ is $x$-sparse.  
            Since $x<\frac{1}{4}$, $S$ has a stable set of size at least $\frac{|S|}{x|S|+1}=\frac{1}{x+\frac{1}{|S|}}\geq \frac{1}{x+x}\geq x^{-\frac{1}{2}}=|G|^{\frac{1}{6c_1}}\geq |G|^c$. 

            We then suppose that the second bullet holds. 
            By taking a vertex from each block, we have a clique or stable set of size $k\geq\epsilon^{-1}=|G|^{\frac{1}{21c_1^2}}\geq |G|^c$. 

            Finally, we suppose that the third bullet holds. 
            So $G$ has a clique or stable set of size at least $(x^{c_1}|F|)^{c_2}\geq (x^{c_1}\epsilon^{2c_1}|G|)^{c_2}=(x^{c_1+\frac{2}{7}}|G|)^{c_2}= |G|^{c_2(1-\frac{c_1+\frac{2}{7}}{3c_1})}\geq |G|^{c_2/2}\geq |G|^{c}$, where the second last inequality is due to $c_1\geq 1$.

            This proves Claim \ref{clm:a large F of G has a complete or anti-complete blockade with small length}. 
        \end{subproof}

        By Claim \ref{clm:a large F of G has a complete or anti-complete blockade with small length} together with Lemma \ref{lem:give a poly-rodl subgraph} (with $a=c_1$), $G$ has an $\epsilon$-restricted induced subgraph $S$ with $|S|\geq \epsilon^{3c_1}|G|$.
        Since $\epsilon=x^{\frac{1}{7c_1}}=|G|^{-\frac{1}{21c_1^2}}$ and $c_1\geq 1$, $|S|\geq \epsilon^{3c_1}|G|=\epsilon^{-21c_1^2+3c_1}\geq \epsilon^{-1}$.
        By taking complement if necessary, we may assume that $S$ is $\epsilon$-sparse. 
        Since $\epsilon\leq \frac{1}{4}$, $S$ has a stable set of size at least $\frac{|S|}{\epsilon|S|+1}=\frac{1}{\epsilon+\frac{1}{|S|}}\geq \frac{1}{\epsilon+\epsilon}\geq \epsilon^{-\frac{1}{2}}=|G|^{\frac{1}{42c_1^2}}\geq |G|^c$.
        
        This completes the proof of Lemma \ref{lem:reduce EH to gene-nice 2}. 
    \end{proof}

\section{Reducing generalized niceness to property $(*)$}\label{sec:iterative-sparsification-1}

    In this section, we prove Lemma \ref{lem:reduce gene-nice to comb-property} which says $\mathcal{F}$ is generalized nice if $\mathcal{F}$ has property $(*)$ and $\mathcal{F}$ is leaf-reducible. 
    
    A finite class of graphs $\mathcal{F}$, we say $\mathcal{F}$ has {\em property $(*)$} if there exist $c_1,c_2,c_3>0$ such that for every $\overline{\mathcal{F}}$-free graph $G$ the following holds. If there is a $(\ell,w)$-comb $((a_i,B_i),i\in[\ell])$ in $G$, where $\ell,w\geq 4$, and there exists $v\in V(G)\setminus(\{a_i:i\in[\ell]\}\cup\bigcup_{i\in[\ell]}B_i)$ such that $v$ is complete to $\bigcup_{i\in[\ell]}B_i$ and anti-complete to $\{a_i:i\in[\ell]\}$, then one of the following holds.
    
    \begin{itemize}
        \item $G$ has a clique or stable set of size at least $w^{c_1}$.
        \item $G$ has a complete or anti-complete $(k,w/k^{c_2})$-blockade with $k\geq \ell^{c_3}$.
        \item $G$ has a pure $(\ell,w/\ell^{2})$-blockade.
    \end{itemize}
    
    We now prove the first lemma through the property $(*)$, while the existence of the comb and the special vertex in the statement of property $(*)$ is given by Lemma \ref{lem:get a comb in a sparse graph}.
    
    \begin{lemma}\label{lem:main-pure-blockade}
        Suppose that $\mathcal{F}$ has property $(*)$ and $\mathcal{F}$ is leaf-reducible. Then
        there exist constants $c_1,c_2,c_3>0$, $c_4,c_5\geq 4$ such that for every $0<x\leq y\leq 2^{-4c_5}$ and every $y^3$-restricted $\overline{\mathcal{F}}$-free graph $G$, one of the following holds.
        \begin{itemize}
        
            \item[$(1)$] There are disjoint $X,Y\subseteq V(G)$ such that $|X|\geq y^{c_4}|G|$, $|Y|\geq (1-c_5y)|G|$, and $Y$ is $x$-sparse or complete to $X$.
            
            \item[$(2)$] $G$ has a $2y^4$-restricted induced subgraph of size at least $y^{c_4}|G|$.
            
            \item[$(3)$] $G$ has a clique or stable set of size at least $(x^{9}|G|)^{c_1}$. 
            
            \item[$(4)$] $G$ has a complete or anti-complete $(k,\frac{|G|}{k^{c_2+\frac{6}{c_3}}})$-blockade with $k\geq y^{-c_3}$.
            
            \item[$(5)$] $G$ has a pure $(\ell,|G|/\ell^{8})$ blockade with $\ell\in[y^{-1}, x^{-2}]$.
            
        \end{itemize}
    \end{lemma}
    
    \begin{proof}
        Let $c_1,c_2,c_3>0$ be given by the property $(*)$. 
        Let $c_4=\max\{4d+1,4\}$ and  $c_5=\max\{h,4\}$, where $d,h$ are given by Lemma \ref{lem:leaf expanding}. 
        We claim that $c_1,c_2,c_3,c_4,c_5$ suffice.

        Suppose first that $\overline{G}$ is $y^3$-sparse.
        By Lemma \ref{lem:leaf expanding} with $b=4$, we have 
        \begin{itemize}
            \item there are disjoint $X,Y\subseteq V(G)$ with $|X|\geq y^{4d+1}$, $|Y|\geq (1-hy)|G|$, such that $Y$ is complete to $X$ in $G$ or
            
            \item $G$ has a $y^4$-restricted induced subgraph of size at least $y^{4d+1}|G|$,  
        \end{itemize}
        which gives the first or the second outcome of Lemma \ref{lem:main-pure-blockade}. 
        
        So we may assume that $G$ is $y^3$-sparse. If $|G|\leq x^{-9}$, then the third outcome of Lemma \ref{lem:main-pure-blockade} holds since $c_1>0$. 
        So we may assume that $|G|\geq x^{-9}\geq y^{-4}$. 
        By Lemma \ref{lem:get a comb in a sparse graph}, one of the following holds.
        \begin{itemize}
        \item There are disjoint $X,Y\subseteq V(G)$ such that $|X|\geq y^4|G|$, $|Y|\geq (1-4y)|G|$, and $Y$ is $x$-sparse to $X$.
        
        \item $G$ is $2y^4$-sparse.
        
        \item For some integer $\ell\in[y^{-1},x^{-2}]$, there is an $(\ell, y^4|G|/\ell^2)$-comb $((a_i,B_i):i\in[\ell])$ in $G$, and a vertex $v\in V(G)\setminus(\{a_i:i\in[\ell]\}\cup\bigcup_{i\in[\ell]}B_i)$ such that $v$ is complete to $\bigcup_{i\in[\ell]}B_i$ and anti-complete to $\{a_i:i\in[\ell]\}$.
        \end{itemize}
        If the first or the second outcome of Lemma \ref{lem:get a comb in a sparse graph} holds, then the first or the second outcome of Lemma \ref{lem:main-pure-blockade} holds. 
        So we may assume that the third outcome of Lemma \ref{lem:get a comb in a sparse graph} holds.         
        Since $y\leq 2^{-4c_5}\leq 2^{-8}$, $\ell\geq y^{-1}> 4$. Moreover,
        \begin{align*}
        w:=y^4|G|/\ell^2 & \geq \max\left\{\frac{|G|}{\ell^6},x^8|G|\right\} & (y^{-1}\leq \ell, \ell\leq x^{-2}, y\geq x) \\
            & \geq x^{-1} & (|G|\geq x^{-9})\\
            & \geq 4 & (x\leq 2^{-4c_5},c_5\geq 4)
        \end{align*}
  
        By the condition that $\mathcal{F}$ has property $(*)$, one of the following holds. 

        \begin{itemize}
            
            \item $G$ has a clique or stable set of size at least $w^{c_1}$.
            
            \item $G$ has a complete or anti-complete $(k,\frac{w}{k^{c_2}})$-blockade with $k\geq \ell^{c_3}$.
            
            \item $G$ has a pure $(\ell,\frac{w}{\ell^{2}})$-blockade.
            
        \end{itemize} 
        
        If the first bullet holds, then since $w^{c_1}\geq (x^8|G|)^{c_1}$, the third outcome of Lemma \ref{lem:main-pure-blockade} holds.
        If the second bullet holds, then since $\frac{w}{k^{c_2}}\geq \frac{|G|}{\ell^6\cdot k^{c_2}}\geq \frac{|G|}{k^{c_2+\frac{6}{c_3}}}$ and $k\geq \ell^{c_3}\geq y^{-c_3}$, the fourth outcome of Lemma \ref{lem:main-pure-blockade} holds. 
        If the third bullet holds, then since $\frac{w}{\ell^2}\geq \frac{|G|}{\ell^8}$, the fifth outcome of Lemma \ref{lem:main-pure-blockade} holds. 

        This completes the proof of Lemma \ref{lem:main-pure-blockade}. 
    \end{proof}
    
    We then present a lemma to translate the first outcome of Lemma \ref{lem:main-pure-blockade} to a long sparse or complete blockade. 

    \begin{lemma}\label{lem:main-once-a-block-turn1}
        Suppose that $\mathcal{F}$ has property $(*)$ and $\mathcal{F}$ is leaf-reducible. 
        Then there exist constants $c_1,c_2,c_3>0$, $c_4,c_5>4$ such that for every $0<x\leq y\leq c=2^{-4c_5}$ and every $cy^3$-restricted $\overline{\mathcal{F}}$-free graph $G$, one of the following holds.
        \begin{itemize}
            \item[$(1)$] $G$ has an $x$-sparse or complete $(y^{-1},y^{c_4+2}|G|)$-blockade.
            
            \item[$(2)$] $G$ has a 2$y^{4}$-restricted induced subgraph $S$ with $|S|\geq y^{c_4+2}|G|$.
            
            \item[$(3)$] $G$ has a clique or stable set of size at least $(x^{10}|G|)^{c_1}$.
            
            \item[$(4)$] $G$ has a complete or anti-complete $(k,\frac{|G|}{k^{c_2+\frac{7}{c_3}}})$-blockade with $k\geq y^{-c_3}$. 
            
            \item[$(5)$] $G$ has a pure $(\ell,\frac{|G|}{\ell^{9}})$ blockade with $\ell\in[y^{-1}, x^{-2}]$. 
        \end{itemize}
    \end{lemma}
    \begin{proof}
        Let $c_1,c_2,c_3>0$, $c_4,c_5\geq 4$ be given by Lemma \ref{lem:main-pure-blockade}. 
        We claim that $c_1,c_2,c_3,c_4,c_5$ suffice. 
        If $|G|\leq y^{-(c_4+2)}$, then the second outcome of Lemma \ref{lem:main-once-a-block-turn1} holds.
        So we may assume that $|G|\geq y^{-(c_4+2)}$. 
        
        Suppose first that for every induced subgraph $F$ of $G$ with $|F|\geq c|G|$, there are disjoint $X,Y\subseteq V(F)$ such that $|X|\geq y^{c_4}|F|$, $|Y|\geq (1-c_5y)|F|$, and $Y$ is $x$-sparse or complete to $X$. 
        By Lemma \ref{lem:get a blockade by iteration} (with $a=c_4$, $b=c_5$), there is an $x$-sparse or complete $(y^{-1},y^{c_4+2}|G|)$-blockade. 
        This give the first outcome of Lemma \ref{lem:main-once-a-block-turn1}. 
        
        So we may assume that there is an induced subgraph $F$ of $G$ with $|F|\geq c|G|$ such that there is no disjoint $X,Y\subseteq V(F)$ such that $|X|\geq y^{c_4}|F|$, $|Y|\geq (1-c_5y)|F|$, and $Y$ is $x$-sparse or complete to $X$. 
        Since $G$ is $cy^3$-restricted and $|F|\geq c|G|$, $F$ is $y^3$-restricted. 
        By Lemma \ref{lem:main-pure-blockade} with $G$ replaced by $F$, one of the following holds.
        \begin{itemize}
    
            \item $F$ has a $2y^4$-restricted induced subgraph of size at least $y^{c_4}|F|$.
            
            \item $F$ has a clique or stable set of size at least $(x^{9}|F|)^{c_1}$. 
            
            \item $F$ has a complete or anti-complete $(k,\frac{|F|}{k^{c_2+\frac{6}{c_3}}})$-blockade with $k\geq y^{-c_3}$.
            
            \item $F$ has a pure $(\ell,|F|/\ell^{8})$ blockade with $\ell\in[y^{-1}, x^{-2}]$.
            
        \end{itemize}
        
        If the first bullet holds, then the second outcome of Lemma \ref{lem:main-once-a-block-turn1} holds since $y^{c_4}|F|\geq y^{c_4}\cdot y|G|=y^{c_4+1}|G|$. If the second bullet holds, then the third outcome of Lemma \ref{lem:main-once-a-block-turn1} holds since $(x^{9}|F|)^{c_1}\geq (x^{9}\cdot x|G|)^{c_1}=(x^{10}|G|)^{c_1}$. If the third bullet holds, then the fourth outcome of Lemma \ref{lem:main-once-a-block-turn1} holds since $\frac{|F|}{k^{c_2+\frac{6}{c_3}}}\geq \frac{y|
        G|}{k^{c_2+\frac{6}{c_3}}}\geq \frac{|G|}{k^{c_2+\frac{6}{c_3}+\frac{1}{c_3}}}=\frac{|G|}{k^{c_2+\frac{7}{c_3}}}$. If the fourth bullet holds, then the fifth outcome of Lemma \ref{lem:main-once-a-block-turn1} holds since $\frac{|F|}{\ell^{8}}\geq \frac{y|G|}{\ell^{8}}\geq \frac{|G|}{\ell^{9}}$.  
        
        This completes the proof of Lemma \ref{lem:main-once-a-block-turn1}.
    \end{proof}
    We then use the method of iterative sparsification to decrease the sparse condition of Lemma \ref{lem:main-once-a-block-turn1} to a constant sparse condition.
   
    \begin{lemma}\label{lem:main-iterative-sparsification-turn1}
        Suppose that $\mathcal{F}$ has property $(*)$ and $\mathcal{F}$ is leaf-reducible. Then 
        there exist constants $c_1,c_2,c_3>0$, $c_4,c_5\ge 4$ such that for every $0<x\leq c^{10}$ and every $c^{10}$-restricted $\overline{\mathcal{F}}$-free graph $G$ where $c=2^{-4c_5}$, one of the following holds.
        \begin{itemize}
            \item[$(1)$] $G$ has an $x$-restricted induced subgraph $S$ of size at least $x^{22c_4}|G|$.
            
            \item[$(2)$] $G$ has a clique or stable set of size at least $(x^{30c_4}|G|)^{c_1}$.
            
            \item[$(3)$] $G$ has a complete or anti-complete $(k,\frac{|G|}{k^{c_2+\frac{27c_4}{c_3}}})$-blockade with $k\geq 2$.
            
            \item[$(4)$] $G$ has an $x$-sparse or pure $(\ell,|G|/\ell^{29c_4})$ blockade with $\ell\in[c^{-1}, x^{-2}]$.  
            
        \end{itemize}
    \end{lemma}
    
    \begin{proof}
        Let $c_1,c_2,c_3>0$, $c_4,c_5\geq 4$ be given by Lemma \ref{lem:main-once-a-block-turn1}. 
        We claim that $c_1,c_2,c_3,c_4,c_5$ suffice. 
        Suppose that none of Lemma \ref{lem:main-iterative-sparsification-turn1} holds. 
        We first present a claim to give the condition of Lemma \ref{lem:give a more sparse subgraph 2}. 
        
        \begin{claim}\label{clm:condition of lemma-give a more sparse subgraph 2}
            For every $y\in[x,c^{3}]$ and every $y^{\frac{10}{3}}$-restricted induced subgraph $F$ of $G$ with $|F|\geq y^{10(c_4+2)}|G|$, there is a $y^{\frac{11}{3}}$-restricted induced subgraph of $F$ with at least $y^{(c_4+2)}|F|$ vertices.
        \end{claim}
        
        \begin{subproof}[Proof of Claim \ref{clm:condition of lemma-give a more sparse subgraph 2}]
            Suppose to the contrary that there is an $y^{\frac{10}{3}}$-restricted induced subgraph $F$ of $G$ with $|F|\geq y^{10(c_4+2)}|G|$ such that $F$ has no $y^{\frac{11}{3}}$-restricted induced subgraph of $F$ with at least $y^{(c_4+2)}|F|$ vertices. 
            Since $y\leq c^3$, $y^{\frac{10}{3}}\leq cy^3$. 
            By Lemma \ref{lem:main-once-a-block-turn1} with $G$ replaced by $F$, one of the following holds.
            \begin{itemize}
                \item $F$ has an $x$-sparse or complete $(y^{-1},y^{c_4+2}|F|)$-blockade. 
            
                \item $F$ has a 2$y^{4}$-restricted induced subgraph $S$ with $|S|\geq y^{c_4+2}|F|$.
            
                \item $F$ has a clique or stable set of size at least $(x^{10}|F|)^{c_1}$.
            
                \item $F$ has a complete or anti-complete $(k,\frac{|F|}{k^{c_2+\frac{7}{c_3}}})$-blockade with $k\geq y^{-c_3}$.
            
                \item $F$ has a pure $(\ell,\frac{|G|}{\ell^{9}})$ blockade with $\ell\in[y^{-1}, x^{-2}]$.
            \end{itemize}

            If the first bullet holds, then the fourth outcome of Lemma \ref{lem:main-iterative-sparsification-turn1} holds since $y^{-1}\in [c^{-1},x^{-1}]$ and $y^{c_4+2}|F|\geq y^{c_4+2}\cdot y^{10(c_4+2)}|G|\geq y^{29c_4}|G|$.

            If the second bullet holds, then it contradicts the condition of Claim \ref{clm:condition of lemma-give a more sparse subgraph 2} since $2y^4\leq y^{\frac{11}{3}}$.

            If the third bullet holds, then the second outcome of Lemma \ref{lem:main-iterative-sparsification-turn1} holds since $x^{10}|F|\geq x^{10}\cdot y^{10(c_4+2)}|G|\geq x^{30c_4}|G|$.  

            Suppose the fourth bullet holds. 
            If $k\geq2$, then the third outcome of Lemma \ref{lem:main-iterative-sparsification-turn1} holds since $\frac{|F|}{k^{c_2+\frac{7}{c_3}}}\geq \frac{y^{10(c_4+2)}|G|}{k^{c_2+\frac{7}{c_3}}}\geq \frac{|G|}{k^{c_2+\frac{7}{c_3}+\frac{10(c_4+2)}{c_3}}}\geq \frac{|G|}{k^{c_2+\frac{27c_4}{c_3}}}$. 
            If $2>k\geq y^{-c_3}> 1$, then since $\frac{|G|}{k^{c_2+\frac{27c_4}{c_3}}}\geq \frac{|G|}{2^{c_2+\frac{27c_4}{c_3}}}$, there is a complete or anti-complete $(2, \frac{|G|}{2^{c_2+\frac{27c_4}{c_3}}})$-blockade, which also gives the third outcome of Lemma \ref{lem:main-iterative-sparsification-turn1}.  

            If the fifth bullet holds, then the fourth outcome of Lemma \ref{lem:main-iterative-sparsification-turn1} holds since $\frac{|F|}{\ell^{9}}\geq \frac{y^{10(c_4+2)}|G|}{\ell^{9}}\geq \frac{|G|}{\ell^{29c_4}}$.  

            This completes the proof of Claim \ref{clm:condition of lemma-give a more sparse subgraph 2}. 
        \end{subproof}
        
        By Claim \ref{clm:condition of lemma-give a more sparse subgraph 2} and Lemma \ref{lem:give a more sparse subgraph 2} (with $b_1=\frac{11}{10}$, $b_2=3(c_4+2)$, $b_3=\frac{3(c_4+2)}{10}$ and $c$ replaced by $c^{10}$, $x$ replaced by $x^{\frac{10}{3}}$, $\lambda$ replaced by $y^{\frac{10}{3}}$, sparse replaced restricted), $G$ has an $x^{\frac{10}{3}}$-restricted induced subgraph $S$ with $|S|\geq (x^{\frac{10}{3}})^{\frac{33(c_4+2)}{10}}|G|= x^{11(c_4+2)}|G|\geq x^{22c_4}|G|$, which gives the first outcome of of Lemma \ref{lem:main-iterative-sparsification-turn1}. 

        This completes the proof of Lemma \ref{lem:main-iterative-sparsification-turn1}. 
    \end{proof}
    
    We then remove the sparsity hypothesis in Lemma \ref{lem:main-iterative-sparsification-turn1} by using R\"odl's Theorem.

    \begin{lemma}\label{main-rodl-turn1}
        Suppose that $\mathcal{F}$ has property $(*)$ and $\mathcal{F}$ is leaf-reducible. Then
        there exist constants $c_1>0,c_4\geq 4$ and $d\geq 58c_4$ such that for every $x\in (0,2^{-d})$ and every $\overline{\mathcal{F}}$-free graph $G$ with $|G|\geq x^{-d}$, one of the following holds. 
        \begin{itemize}
            \item[$(1)$] $G$ has an $x$-restricted induced subgraph $S$ of size at least $x^{23c_4}|G|$. 
            
            \item[$(2)$] $G$ has a pure or $x$-sparse $(k,|G|/k^d)$-blockade with $k\in[2,x^{-1}]$.
            
            \item[$(3)$] $G$ has a clique or stable set of size at least $(x^{31c_4}|G|)^{c_1}$.
            
            \item[$(4)$] $G$ has a complete or anti-complete $(k,|G|/k^d)$-blockade with $k\geq x^{-1}$.  

        \end{itemize}
    \end{lemma}
    
    \begin{proof}
        Let $c_1,c_2,c_3>0$, $c_4,c_5\geq 5$ be given by Lemma \ref{lem:main-iterative-sparsification-turn1}, $c=2^{-4c_5}$ and $\xi=c^{10}$. 
        By Theorem \ref{thm:Rodl} (with $\epsilon=\xi$), there exists $0<\delta\leq 1$ such that every $\overline{\mathcal{F}}$-free graph has an $\xi$-restricted induced subgraph of size at least $\delta|G|$. 
        Let $d$ be sufficiently large so that $2^{-d}<c^{10}$, $\delta\geq 2^{58c_4-d}$ (which implies $d\geq 58c_4$), $\delta\geq 2^{c_2+\frac{27c_4}{c_3}-d}$.
        In the following we prove that $d$ and $c_1,c_4$ satisfy the lemma.

        Fix an $\overline{\mathcal{F}}$-free graph $G$ with $|G|\geq x^{-d}$.
        Since $G$ is $\overline{\mathcal{F}}$-free, $G$ has an $\xi$-restricted induced subgraph $F$ of size at least $\delta|G|$. 
        Since $x\leq 2^{-d}< c^{10}$, it follows from Lemma \ref{lem:main-iterative-sparsification-turn1} with $G$ replaced by $F$ that one of the following holds.
        
        \begin{itemize}
            \item $F$ has an $x$-restricted induced subgraph $S$ of size at least $x^{22c_4}|F|$.
            
            \item $F$ has a clique or stable set of size at least $(x^{30c_4}|F|)^{c_1}$.
            
            \item $F$ has a complete or anti-complete $(k,\frac{|F|}{k^{c_2+\frac{27c_4}{c_3}}})$-blockade with $k\geq 2$.
            
            \item $F$ has an $x$-sparse or pure $(\ell,|F|/\ell^{29c_4})$ blockade with $\ell\in[c^{-1}, x^{-2}]$.

        \end{itemize}

        If the first bullet holds, then the first outcome of Lemma \ref{main-rodl-turn1} holds since $x^{22c_4}|F|\geq x^{22c_4}\cdot \delta|G|\geq x^{22c_4}\cdot 2^{58c_4-d}|G|\geq x^{23c_4}|G|$, where the last inequality is due to $2^{-d}>x$.

        If the second bullet holds, then the third outcome of Lemma \ref{main-rodl-turn1} holds since $x^{30c_4}|F|\geq x^{30c_4}\cdot \delta|G|\geq x^{30c_4}\cdot 2^{58c_4-d}|G|\geq x^{31c_4}|G|$.

        If the third bullet holds, then the second or the fourth outcome of Lemma \ref{main-rodl-turn1} holds since $\frac{|F|}{k^{c_2+\frac{27c_4}{c_3}}} \geq\frac{\delta|G|}{k^{c_2+\frac{27c_4}{c_3}}}\geq \frac{|G|}{k^{c_2+\frac{27c_4}{c_3}}\cdot 2^{d-(c_2+\frac{27c_4}{c_3})}}\geq \frac{|G|}{k^d}$, where the second inequality is due to $\delta\geq 2^{c_2+\frac{27c_4}{c_3}-d}$.

        If the fourth bullet holds, then the second outcome holds with $k=\ell^{1/2}\in[2,x^{-1}]$, since $\frac{|F|}{\ell^{29c_4}}\geq \frac{\delta|G|}{(\ell^{1/2})^{58c_4}}\geq \frac{|G|}{(\ell^{1/2})^{58c_4}\cdot 2^{d-58c_4}}\geq \frac{|G|}{(\ell^{1/2})^{d}}$. 

        This completes the proof of Lemma \ref{main-rodl-turn1}. 
    \end{proof}
    Now we are ready to prove Lemma \ref{lem:reduce gene-nice to comb-property}, which we restate here for readers' convenience. 

    \begin{lemma}\label{lem:main-impove-blockade-long-turn1}
        Suppose that $\mathcal{F}$ has property $(*)$ and $\mathcal{F}$ is leaf-reducible.  
        Then $\mathcal{F}$ is generalized nice, that is, 
        there exist constants $c_1\geq 3$, $c_2\geq 8$, $c_3,c_4,c_5,c_8>0$, $c_7\ge 4$ and $c_6\geq 1$ such that for every $\overline{\mathcal{F}}$-free graph $G$ and every $\epsilon\in(0,\frac{1}{2})$, either  
        \begin{itemize}
            \item[$(1)$] $G$ has an $(\epsilon^{-1},  \epsilon^{c_1}|G|)$-blockade whose blocks are pairwise complete or weakly $\epsilon^{c_2}$-sparse; 

            \item[$(2)$] $G$ has a clique or stable set of size at least $(\epsilon^{c_3}|G|)^{c_4}$;

            \item[$(3)$] $G$ has a complete or anti-complete $(k,|G|/k^{c_5})$-blockade with $k\geq \epsilon^{-c_6}$; 

            \item[$(4)$] $F$ has an $\epsilon^{c_7}$-restricted induced subgraph $S$ of size at least $\epsilon^{c_8}|G|$. 
        \end{itemize}
        
    \end{lemma}
    
    \begin{proof}
        Choose $d\geq 58$ and $c'_1>0,c_4'\geq 4$ to satisfy Lemma \ref{main-rodl-turn1} (with $c_1$ replaced by $c_1'$ and $c_4$ replaced by $c'_4$). 
        We claim that $c_1=10d^2, c_2=d,c_3=156c'_4d,c_4=c'_1,c_5=2d$, $c_6=5d$, $c_7=5d$ and $c_8=116c_4'd$ suffice. 
        Let $x=\epsilon^{5d}$. 
        Suppose that none of the outcomes holds. 
        If $|G|\leq \epsilon^{-1}$, then the second outcome of Lemma \ref{lem:main-impove-blockade-long-turn1} holds. 
        If $\epsilon^{-1}\leq |G|\leq \epsilon^{-10d^2}$, then the first outcome of Lemma \ref{lem:main-impove-blockade-long-turn1} holds. 
        So $|G|\geq \epsilon^{-10d^2}$.
        We then present a claim to give some property of an induced subgraph of $G$ with at least $\epsilon^{d}|G|$ vertices.  
        \begin{claim}\label{clm:a large F of G has a complete or anti-complete blockade with small length 2}
            Every induced subgraph $F$ of $G$ of size at least $\epsilon^{d}|G|$ has a pure or $x$-sparse $(k,|F|/k^d)$-blockade with $k\in[2,x^{-1}]$.  
        \end{claim}
        \begin{subproof}[Proof of Claim \ref{clm:a large F of G has a complete or anti-complete blockade with small length 2}]
            Since $|G|\geq \epsilon^{-10d^2}$, $|F|\geq \epsilon^d|G|\geq \epsilon^{-10d^2+d}\geq \epsilon^{-5d^2}=x^{-d}$. 
            Note that $x\in (0,2^{-d})$. 
            By Lemma \ref{main-rodl-turn1}, one of the following holds. 
            
            \begin{itemize}
                \item $F$ has an $x$-restricted induced subgraph $S$ of size at least $x^{23c'_4}|F|$.
            
                \item $F$ has a pure or $x$-sparse $(k,|F|/k^d)$-blockade with $k\in[2,x^{-1}]$. 
            
                \item $F$ has a clique or stable set of size at least $(x^{31c'_4}|F|)^{c_1}$.  
            
                \item $F$ has a complete or anti-complete $(k,|F|/k^d)$-blockade with $k\geq x^{-1}$. 
                
            \end{itemize}

            If the first bullet holds, then the fourth outcome of Lemma \ref{lem:main-impove-blockade-long-turn1} holds since $x^{23c'_4}|F|\geq x^{23c'_4}\cdot \epsilon^d|G|\geq \epsilon^{116c_4'd}|G|$. 

            If the third bullet holds, then the second outcome of Lemma \ref{lem:main-impove-blockade-long-turn1} holds since $x^{31c'_4}|F|\geq \epsilon^{155c'_4d}\cdot \epsilon^d|G|\geq \epsilon^{156c'_4d}|G|$. 
            
            If the fourth bullet holds, then the third outcome of Lemma \ref{lem:main-impove-blockade-long-turn1} holds since $\frac{|F|}{k^d}\geq \frac{x^{1/5}|G|}{k^d}\geq \frac{k^{-1/5}|G|}{k^d}\geq \frac{|G|}{k^{2d}}$. 

            So the second bullet holds, which proves Claim \ref{clm:a large F of G has a complete or anti-complete blockade with small length 2}. 
        \end{subproof}
        By Claim \ref{clm:a large F of G has a complete or anti-complete blockade with small length 2} and Lemma \ref{lem:pure or sparse to pairwise complete or weakly sparse}, $G$ has an $(\epsilon^{-1}, \epsilon^{10d^2}|G|)$-blockade whose blocks are pairwise complete or weakly $\epsilon^d$-sparse, which gives the first outcome of Lemma \ref{lem:main-impove-blockade-long-turn1}. 

        This completes the proof of Lemma \ref{lem:main-impove-blockade-long-turn1}. 
    \end{proof}

\section{Deducing property $(*)$}\label{sec:proving (*)}
    In this section, we prove Lemma \ref{lem:proving (*)} which we restate for reader's convenience.
    
    \begin{lemma}\label{lem:property of co-E}
       
        Let $\mathcal{F}_1$ and $\mathcal{F}_2$ be two finite sets of graphs that satisfy the Erd\H{o}s-Hajnal property. 
        Let $\mathcal{H}$ be a finite set of graphs.
        Suppose that for every $\overline{\mathcal{H}}$-free graph $G$ and every $(\ell,w)$-comb $((a_i,B_i),i\in[\ell])$ in $G$ with $\ell,w\geq 4$, if there exists a special vertex $v\in V(G)\setminus(\{a_i:i\in[\ell]\}\cup\bigcup_{i\in[\ell]}B_i)$ such that $v$ is complete to $\bigcup_{i\in[\ell]}B_i$ and anti-complete to $\{a_i:i\in[\ell]\}$,
        $B_i$ (for each $i\in[\ell]$) can be partitioned into $X_i,Y_i$ such that 
        \begin{itemize}
            \item[$(1)$] $Y_i$ is $\mathcal{F}_1$-free;  
            \item[$(2)$] $X_i$ can be partitioned into $(A^i_1,\ldots, A^i_{t_i})$ such that 
            \begin{itemize}
                \item[$(2.1)$] $(A^i_1,\ldots, A^i_{t_i})$ is a pure blockade; 
                \item[$(2.2)$] the pattern of $(A^i_1,\ldots, A^i_{t_i})$, whose vertex set consists of all blocks of $(A^i_1,\ldots, A^i_{t_i})$ and two vertices are adjacent if and only if their corresponding blocks are complete to each other, is $\mathcal{F}_2$-free; 
                \item[$(2.3)$] for each $j\in [t_i]$ and every vertex $u\in \bigcup_{k\in [\ell]\setminus \{i\}} B_k$, $u$ is pure to $A^i_j$.
            \end{itemize}
        \end{itemize} 
        Then $\mathcal{H}$ satisfies property $(*)$.
    \end{lemma}

    \begin{proof}
    We need to show that
    there exist $c_1,c_2,c_3>0$ such that the following holds for every $\overline{\mathcal{H}}$-free graph $G$. 
    If there is a $(\ell,w)$-comb $((a_i,B_i),i\in[\ell])$ in $G$ with $\ell,w\geq 4$, and there exists $v\in V(G)\setminus(\{a_i:i\in[\ell]\}\cup\bigcup_{i\in[\ell]}B_i)$ such that $v$ is complete to $\bigcup_{i\in[\ell]}B_i$ and anti-complete to $\{a_i:i\in[\ell]\}$, then one of the following holds.
    
        \begin{itemize}
            
            \item[$(a)$] $G$ has a clique or stable set of size at least $w^{c_1}$.
            
            \item[$(b)$] $G$ has a complete or anti-complete $(k,w/k^{c_2})$-blockade with $k\geq \ell^{c_3}$.
            
            \item[$(c)$] $G$ has a pure $(\ell,w/\ell^{2})$-blockade.

        \end{itemize}
        
        Let $c=\min\{c(\mathcal{F}_1),c(\mathcal{F}_2)\}$ where $c(\mathcal{F}_i)$ is the Erd\H{o}s-Hajnal constant for $\mathcal{F}_i$.
        We claim that $c_1=\frac{c}{2}$, $c_2=\frac{10}{c}$ and $c_3=\frac{c}{4}$ suffice. 
        Suppose that none of the outcomes holds.
        Since $|B_i|\geq w$ for each $i\in [\ell]$, either $|X_i|\geq \frac{w}{2}$ or $|Y_i|\geq \frac{w}{2}$. 
        Suppose first that $|Y_i|\geq \frac{w}{2}$ for some $i\in [\ell]$. 
        Since $Y_i$ is $\mathcal{F}_1$-free, $Y_i$ has a clique or stable set with at least $(\frac{w}{2})^c$ vertices.
        Since $w\geq 4$, $(\frac{w}{2})^c\geq w^{\frac{c}{2}}=w^{c_1}$. 
        So $(a)$ holds. 
        Thus, we may assume that $|X_i|\geq \frac{w}{2}$ for each $i\in [\ell]$.  
        \begin{claim}\label{clm:upbound of block A}
            There exists some index $i\in [\ell]$ such that every block of the pure blockade $(A^i_1,\ldots, A^i_{t_i})$ has size at most $\frac{w}{2\ell}$. 
        \end{claim}
        \begin{subproof}[Proof of Claim \ref{clm:upbound of block A}]
            Suppose to the contrary that there exists a block of $(A^i_1,\ldots, A^i_{t_i})$, say $A_1^i$, with at least $\frac{w}{2\ell}$ vertices for every $i\in [\ell]$. 
            By (2.3), $(A_1^1,\ldots, A_1^\ell)$ is a pure $(\ell, \frac{w}{2\ell})$-blockade. 
            Since $\ell\geq 4$, $2\ell\leq \ell^2$ and so
            $(A_1^1,\ldots, A_1^\ell)$ is also a pure $(\ell, \frac{w}{\ell^2})$-blockade. 
            So $(c)$ holds, a contradiction. 
        \end{subproof}
        By Claim \ref{clm:upbound of block A}, there is an index $i\in [\ell]$ such that $|A_j^i|\leq \frac{w}{2\ell}$ for every $j\in [t_i]$. 
        Since $|X_i|\geq \frac{w}{2}$, $t_i\geq \ell$. 
        Without loss of generality, we may assume that $|A_1^i|\geq |A_2^i|\geq \ldots\geq |A_{t_i}^i|$. 
        Let $q=\lceil \log_{\ell^{1/2}}t_i\rceil$. 
        Let $S_1=\bigcup_{j\leq \ell^{1/2}}A_j^i$, and $$S_k=\left(\bigcup_{j\leq \min\{\ell^{k/2},t_i\}}A^i_j\right)\setminus \left(\bigcup_{k'<k}S_{k'}\right) \textnormal{ for } k=2,\ldots,q.$$ 
        Then $|X_i|=\sum_{k=1}^q |S_k|$.
        \begin{claim}\label{clm:upbound of block in S}
            For each $k\in [q-1]$, there exists a block of $S_k$ of size less than $\frac{w}{\ell^{5k/2}}$. 
        \end{claim}
        \begin{subproof}[Proof of Claim \ref{clm:upbound of block in S}]
            Suppose to the contrary that every block of $S_k$ has at least $\frac{w}{\ell^{5k/2}}$ vertices.  
            Since $\ell, w\geq 4$, $\lfloor \ell^{k/2}\rfloor^c\geq \ell^{ck/4}$.
            Since the pattern of $(A^i_1,\ldots, A^i_{t_i})$, whose vertex set consists of all blocks of $(A^i_1,\ldots, A^i_{t_i})$ and two vertices are adjacent if and only if their corresponding blocks are complete to each other, is $\mathcal{F}_2$-free, there is a complete or anti-complete $(\ell^{ck/4}, \frac{w}{\ell^{5k/2}})$-blockade in $\bigcup_{j\leq k} S_j$.
            Since $c_2=\frac{10}{c}$, this blockade is also a $(\ell^{ck/4}, \frac{w}{(\ell^{ck/4})^{c_2}})$-blockade.
            Then $(b)$ holds, a contradiction. 
        \end{subproof}
        By Claim \ref{clm:upbound of block in S}, for each $k\in[q-1]$, there is a block in $S_k$ of size less than $\frac{w}{\ell^{5k/2}}$. 
        So every block in $S
        _{k+1}$ has at most $\frac{w}{\ell^{5k/2}}$ vertices for each $k\in[q-1]$. 
        Thus, 
        \begin{align*}
            |X_i|&=\sum_{k=1}^q |S_k|\\
                 &\leq |S_1|+\sum_{k=1}^{q-1} (\ell^{(k+1)/2})\cdot \frac{w}{\ell^{5k/2}}\\
                 &\leq \ell^{1/2}\cdot \frac{w}{2\ell}+\sum_{k=1}^{q-1} (\ell^{(k+1)/2})\cdot \frac{w}{\ell^{5k/2}}\\
                 &= w(\frac{1}{2\ell^{1/2}}+\sum_{k=1}^{q-1} (\frac{1}{\ell^{1/2}})^{5k-k-1})\\
                 &\leq w(\frac{1}{4}+\sum_{k=1}^{q-1}(\frac{1}{2})^{k+2})\\
                 &<\frac{w}{2}. 
        \end{align*}
        This contradicts the fact that $|X_i|\geq \frac{w}{2}$. 
        This completes the proof of Lemma \ref{lem:property of co-E}.
    \end{proof}

    \noindent {\bf Remark.} If we set $\mathcal{F}_2=\{2K_1\}$, $Y_i=\emptyset$ and partition $X_i$ into anti-connected components, both $P_5$ and bull satisfy the conditions of Lemma \ref{lem:property of co-E}. Moreover, it is easy to see that each of $P_5$ and bull is leaf-reducible and wonderful. 
    This already proves that Erd\H{o}s-Hajnal conjecture holds for $P_5$ and bull.
    
    \section{Completing the proof}\label{sec:deducing structure}

In this section, we prove that co-$E$-free graphs and co-Bird-free graphs satisfy the hypothesis of Lemma \ref{lem:property of co-E}, which implies that $E$-graph and Bird satisfy the Erd\H{o}s-Hajnal property. 
We first define the quotient blockade of a blockade and derive some useful properties of quotient blockades.

Suppose that $\mathcal{L}=(L_1,\ldots, L_n)$ is a blockade. 
        We say $L_i$ and $L_j$ have {\em relation $\mathcal{M}$} if and only if $i= j$ or there is a block sequence $L_i=L_{r_1},L_{r_2},\ldots,L_{r_m}=L_{j}$ such that for each $k\in [m-1]$, $L_{r_{k}}$ and $L_{r_{k+1}}$ are mixed. 
        It is easy to check that $\mathcal{M}$ is an equivalence relation. Let $\mathcal{L}/\mathcal{M}$ be the {\em quotient blockade} of $\mathcal{L}$ where each block of $\mathcal{L}/\mathcal{M}$ is the union of all blocks in an equivalence class of $\mathcal{M}$.
        We present two properties of the quotient blockade. 
        \begin{lemma}\label{clm:property of levels}
            Suppose that $\mathcal{L}$ is a blockade such that each block of $\mathcal{L}$ is connected (anti-connected). 
            Then  
            \begin{itemize}
                \item[$(1)$] Each block of $\mathcal{L}/\mathcal{M}$ is also connected (resp. anti-connected).
                
                \item[$(2)$] Each two blocks of $\mathcal{L}$ that are contained in different blocks of $\mathcal{L}/\mathcal{M}$ are pure to each other.
                
                \item[$(3)$] If there are a block $D$ of $\mathcal{L}/\mathcal{M}$ and a vertex $u\notin D$ that is mixed on $D$ but is pure to each block of $\mathcal{L}$ contained in $D$, then there are two mixed blocks of $\mathcal{L}$ contained in $D$ such that $u$ is complete to one of the blocks and anti-complete to the other.
    
            \end{itemize}
        \end{lemma}
        
        \begin{proof}
            The first two statements of Lemma \ref{clm:property of levels} follow immediately from the definition of $\mathcal{L}/\mathcal{M}$. 
            We now prove (3). Since $u$ is mixed on $D$ but is pure to each block of $\mathcal{L}$ contained in $D$, there are two blocks of $\mathcal{L}$, say $A_1,A_2$, that are contained in $D$ such that $u$ is complete to $A_1$ and anti-complete to $A_2$. 
            Since $D$ is an equivalence class of $\mathcal{L}$ with respect to $\mathcal{M}$, there is a block sequence $A_1=A_{r_1},A_{r_2},\ldots,A_{r_m}=A_{2}$ such that for each $k\in [m-1]$, $A_{r_{k}}$ and $A_{r_{k+1}}$ are mixed. 
            So there is an index $k$ such that $u$ is complete to $A_{r_{k}}$ and anti-complete to $A_{r_{k+1}}$.  
            This proves Lemma \ref{clm:property of levels} (3). 
        \end{proof}

        \begin{lemma}\label{clm:inductive}
            Suppose that $\mathcal{L}$ is a blockade in $G$ such that each block of $\mathcal{L}$ is connected (anti-connected). 
            If there are two mixed blocks $D_1,D_2$ of $\mathcal{L}/\mathcal{M}$ and three vertices $x,y,u\notin D_1\cup D_2$ such that 
            \begin{itemize}
                \item $x,y$ are two non-adjacent vertices that are complete to $D_1\cup D_2$,
                \item $u\in N(x)\setminus N(y)$ is complete to $D_1$ but is anti-complete to $D_2$, and
                \item no vertex in $D_1$ can mix on $D_2$,
            \end{itemize}
            then there are two mixed blocks $A_1,A_2$ of $\mathcal{L}$ that contained in $D_1$ and three vertices $x',y',u'\notin A_1\cup A_2$ such that 
            \begin{itemize}
                \item[$(1)$] $x',y'$ are two non-adjacent vertices that are complete to $A_1\cup A_2$ and
                \item[$(2)$] $u'\in N(x')\setminus N(y')$ is complete to $A_1$ but is anti-complete to $A_2$.  
            \end{itemize}
        \end{lemma}
        
        \begin{proof}
            Since $D_1$ and $D_2$ are mixed, either there is a vertex $b_1\in D_1$ mixed on $D_2$ or there is a vertex $b_2\in D_2$ mixed on $D_1$. By our assumption, 
            there is a vertex $b_2\in D_2$ mixed on $D_1$. 
            By Lemma \ref{clm:property of levels} (2) and (3), there are two mixed blocks of $\mathcal{L}$, say $A_1,A_2$, contained in $D_1$ such that $b_2$ is complete to $A_1$ but is anti-complete to $A_2$. 
            Then $(x',y',u')=(y,u,b_2)$ are the desired vertices.
        \end{proof}

    \subsection{E-graph}

    In this subsection, we prove that co-$E$-graph-free graphs satisfy the hypothesis of  Lemma \ref{lem:property of co-E}.

    Let {\em chair} be the 5-vertex graph obtained from an induced $P_4$ by adding a new vertex pendent to the second vertex of the $P_4$ and {\em co-chair} be the complement of the chair graph.
    Let $H_0$ be a 5-wheel labeled as in Figure \ref{fig:P7-star graphs}. 
    For each integer $1\leq i\leq 5$, let $H_i$ be the graph obtained from $H_{i-1}$ by adding a vertex $v_i'$ and an edge $v_iv_i'$, see Figure \ref{fig:P7-star graphs} for an illustration of chair, $H_0,H_1$ and $H_5$. 
    \begin{figure}[h!]
        \centering
        \begin{tikzpicture}[scale=0.6]
        \tikzstyle{vertex}=[circle, draw, fill=white, inner sep=1pt, minimum size=5pt]
        
            \node[vertex](1) at (-.6,-1.1){\scriptsize $v_1$};
            \node[vertex](2) at (-.6,0){\scriptsize $v_2$};
            \node[vertex](3) at (.6,0){\scriptsize $v_3$};
            \node[vertex](4) at (.6,1.1){\scriptsize $v_4$};
            \node[vertex](5) at (.6,-1.1){\scriptsize $v_5$};

            \node at (-90:3) {The chair graph};
            
            \foreach \from/\to in {1/2,2/3,3/4,3/5}
            \draw (\from) -- (\to);

        \end{tikzpicture}
        \hspace{10mm}
        \begin{tikzpicture}[scale=0.6]
        \tikzstyle{vertex}=[circle, draw, fill=white, inner sep=1pt, minimum size=5pt]
        
            \node[vertex](0) at (0:0){\scriptsize $w$};
            
            \node[vertex](1) at (90:1.2){\scriptsize $v_1$};
            \node[vertex](2) at (90-72:1.2){\scriptsize $v_2$};
            \node[vertex](3) at (90-144:1.2){\scriptsize $v_3$};
            \node[vertex](4) at (90-216:1.2){\scriptsize $v_4$};
            \node[vertex](5) at (90+72:1.2){\scriptsize $v_5$};

            \node at (-90:3) {$H_0$};
            
            \foreach \from/\to in {0/1,0/2,0/3,0/4,0/5,1/2,2/3,3/4,4/5,1/5}
            \draw (\from) -- (\to);

        \end{tikzpicture}
        \hspace{10mm}
        \begin{tikzpicture}[scale=0.6]
        \tikzstyle{vertex}=[circle, draw, fill=white, inner sep=1pt, minimum size=5pt]
        
            \node[vertex](0) at (0:0){\scriptsize $w$};
            
            \node[vertex](1) at (90:1.2){\scriptsize $v_1$};
            \node[vertex](2) at (90-72:1.2){\scriptsize $v_2$};
            \node[vertex](3) at (90-144:1.2){\scriptsize $v_3$};
            \node[vertex](4) at (90-216:1.2){\scriptsize $v_4$};
            \node[vertex](5) at (90+72:1.2){\scriptsize $v_5$};

            \node[vertex](11) at (90:2.2){\scriptsize $v_1'$};
            
            \node at (-90:3) {$H_1$};
            
            \foreach \from/\to in {0/1,0/2,0/3,0/4,0/5,1/2,2/3,3/4,4/5,1/5}
            \draw (\from) -- (\to);

            \foreach \from/\to in {1/11}
            \draw (\from) -- (\to);
        \end{tikzpicture}
        \hspace{10mm}
        \begin{tikzpicture}[scale=0.6]
        \tikzstyle{vertex}=[circle, draw, fill=white, inner sep=1pt, minimum size=5pt]
        
            \node[vertex](0) at (0:0){\scriptsize $w$};
            
            \node[vertex](1) at (90:1.2){\scriptsize $v_1$};
            \node[vertex](2) at (90-72:1.2){\scriptsize $v_2$};
            \node[vertex](3) at (90-144:1.2){\scriptsize $v_3$};
            \node[vertex](4) at (90-216:1.2){\scriptsize $v_4$};
            \node[vertex](5) at (90+72:1.2){\scriptsize $v_5$};

            \node[vertex](11) at (90:2.2){\scriptsize $v_1'$};
            \node[vertex](21) at (90-72:2.2){\scriptsize $v_2'$};
            \node[vertex](31) at (90-144:2.2){\scriptsize $v_3'$};
            \node[vertex](41) at (90-216:2.2){\scriptsize $v_4'$};
            \node[vertex](51) at (90+72:2.2){\scriptsize $v_5'$};
            
            \node at (-90:3) {$H_5$};
            
            \foreach \from/\to in {0/1,0/2,0/3,0/4,0/5,1/2,2/3,3/4,4/5,1/5}
            \draw (\from) -- (\to);

            \foreach \from/\to in {1/11,2/21,3/31,4/41,5/51}
            \draw (\from) -- (\to);
        \end{tikzpicture}
        \caption{The chair graph, $H_0,H_1$ and $H_5$.}
        \label{fig:P7-star graphs}
    \end{figure}
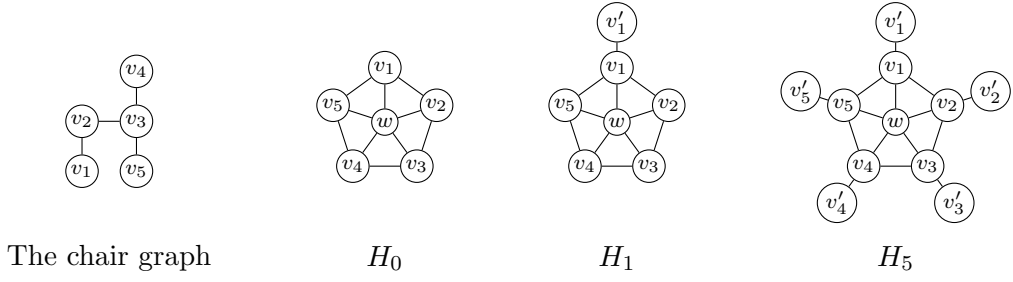

    Let us first record the following simple lemma. 
    \begin{lemma}\label{lem:new EH}
        $\{H_5, \text{co-}E\}$ has the Erd\H{o}s-Hajnal property. 
    \end{lemma}
    
    \begin{proof}
        By Theorems \ref{thm:E-H-substitute} and \ref{thm:C5 has EH property}, $H_0$ and co-chair have the Erd\H{o}s-Hajnal property. 
        It follows that $\{H_0,\text{co-}E\}$ and $\{H_i,\text{co-chair}\}$ have the Erd\H{o}s-Hajnal property for every $i\in [5]$. 
        Assume that $\{H_{i-1},\text{co-}E\}$ has the Erd\H{o}s-Hajnal property for $1\leq i\leq 5$. 
        By Corollary \ref{coro:degree 1 and degree n-2} with $F_1=H_i$ and $F_2$ being the $E$-graph, $\{H_i, \text{co-}E\}$ has the Erd\H{o}s-Hajnal property.  
        So $\{H_5, \text{co-}E\}$ has the Erd\H{o}s-Hajnal property. 
    \end{proof}

\begin{lemma}\label{lem:coro-lem for E}
        Let $G$ be a co-$E$-free graph with a $(\ell,w)$-comb $((a_i,B_i),i\in[\ell])$ in $G$ and a vertex $v\in V(G)\setminus(\{a_i:i\in[\ell]\}\cup\bigcup_{i\in[\ell]}B_i)$ such that $v$ is complete to $\bigcup_{i\in[\ell]}B_i$ and anti-complete to $\{a_i:i\in[\ell]\}$. 
        For each $i\in[\ell]$, $B_i$ can be partitioned into $X_i,Y_i$ such that 
        \begin{itemize}
            \item[$(1)$] $Y_i$ is $\{H_5, \text{co-}E\}$-free;  
            \item[$(2)$] $X_i$ can be partitioned into $(A^i_1,\ldots, A^i_{t_i})$ such that 
            \begin{itemize}
                \item[$(2.1)$] $(A^i_1,\ldots, A^i_{t_i})$ is a pure blockade; 
                \item[$(2.2)$] the pattern of $(A^i_1,\ldots, A^i_{t_i})$, whose vertex set consists of all blocks of $(A^i_1,\ldots, A^i_{t_i})$ and two vertices are adjacent if and only if their corresponding blocks are complete to each other, is $\{H_5, \text{co-}E\}$-free; 
                \item[$(2.3)$] for each $j\in [t_i]$ and every vertex $u\in \bigcup_{k\in [\ell]\setminus \{i\}} B_k$, $u$ is pure to $A^i_j$.
            \end{itemize}
        \end{itemize} 
    \end{lemma}
    \begin{proof}
        We begin with a useful claim. 
        \begin{claim}\label{clm:a vertex mixed on a path via E}
            Let $P$ be an induced path in $G$ and $x,y\in V(G)\setminus V(P)$ be two non-adjacent vertices complete to $P$. 
            For each vertex $u\in N(x)\setminus N(y)$ that is mixed on $P$, 
            \begin{itemize}
                \item[$(1)$] $u$ has no two consecutive non-neighbors in $P$; and 
                \item[$(2)$] $u$ has no three consecutive neighbors in $P$.  
            \end{itemize} 
        \end{claim}
        \begin{subproof}[Proof of Claim \ref{clm:a vertex mixed on a path via E}]
            Let $u\in N(x)\setminus N(y)$ be a vertex that is mixed on $P$. 
            Suppose first that $u$ has two consecutive non-neighbors in $P$. 
            So there is an induced subpath $a-b-c$ of $P$ such that $u$ is adjacent to $a$ but is not adjacent to $b,c$. 
            Then $\{x,y,u,a,b,c\}$ induces a co-$E$ (see Figure \ref{fig:co-E in claim 5.2.1}), a contradiction. 
            This proves Claim \ref{clm:a vertex mixed on a path via E} (1). 

            Now we suppose that $u$ has three consecutive neighbors in $P$. 
            So there is an induced subpath $a-b-c-d$ of $P$ such that $u$ is adjacent to $a,b,c$ but is not adjacent to $d$. 
            Then $\{y,u,a,b,c,d\}$ induces a co-$E$ (see Figure \ref{fig:co-E in claim 5.2.1}), a contradiction. 
            This proves Claim \ref{clm:a vertex mixed on a path via E} (2). 
        \end{subproof}
        \begin{figure}[h!]
        \centering
        \begin{tikzpicture}[scale=1]
        \tikzstyle{vertex}=[circle, draw, fill=white, inner sep=1pt, minimum size=5pt]
        
            \node[vertex](1) at (0,1){\scriptsize $a$};
            \node[vertex](2) at (0,0){\scriptsize $b$};
            \node[vertex](3) at (0,-1){\scriptsize $c$};
            
            \node[vertex](4) at (1,0){\scriptsize $x$};
            \node[vertex](5) at (-1,0){\scriptsize $y$};
            \node[vertex](6) at (1,1){\scriptsize $u$};

            \node at (0,-2.5) {Proof of Claim \ref{clm:a vertex mixed on a path via E} (1)};
            
            \foreach \from/\to in {1/2,2/3}
            \draw[blue] (\from) -- (\to);

            \foreach \from/\to in {4/1,4/2,4/3}
            \draw (\from) -- (\to);

            \foreach \from/\to in {5/1,5/2,5/3}
            \draw (\from) -- (\to);

            \foreach \from/\to in {6/1,6/4}
            \draw (\from) -- (\to);

        \end{tikzpicture}
        \hspace{10mm}
        \begin{tikzpicture}[scale=1]
        \tikzstyle{vertex}=[circle, draw, fill=white, inner sep=1pt, minimum size=5pt]
        
            \node[vertex](1) at (0,1.5){\scriptsize $a$};
            \node[vertex](2) at (0,0.5){\scriptsize $b$};
            \node[vertex](3) at (0,-0.5){\scriptsize $c$};
            \node[vertex](4) at (0,-1.5){\scriptsize $d$};

            \node[vertex](5) at (-1,0){\scriptsize $y$};
            \node[vertex](6) at (1,1){\scriptsize $u$};

            \node at (0,-2.5) {Proof of Claim \ref{clm:a vertex mixed on a path via E} (2)};
            
            \foreach \from/\to in {1/2,2/3,3/4}
            \draw[blue] (\from) -- (\to);

            \foreach \from/\to in {5/1,5/2,5/3,5/4}
            \draw (\from) -- (\to);

            \foreach \from/\to in {6/1,6/2,6/3}
            \draw (\from) -- (\to);

        \end{tikzpicture}
        \caption{Induced co-$E$s in the proof of Claim \ref{clm:a vertex mixed on a path via E} where the blue line is the subpath of $P$.}
        \label{fig:co-E in claim 5.2.1}
    \end{figure}
        By Claim \ref{clm:a vertex mixed on a path via E}, we obtain a ``homogeneous structure'' in $B_i$. 
        Formally, we have the following. 
        \begin{claim}\label{clm:homo structure for E}
            If $x,y$ are two non-adjacent vertices complete to an induced $H_5$,
            then every vertex $u\in N(x)\setminus N(y)$ is pure to this $H_5$. 
            In particular, for every vertex $u\in \bigcup_{k\in [\ell]\setminus \{i\}} B_k$, $u$ is pure to every induced $H_5$ in $B_i$. 
        \end{claim}
        
        \begin{subproof}[Proof of Claim \ref{clm:homo structure for E}]
            Suppose to the contrary that there is a vertex $u\in N(x)\setminus N(y)$ that is mixed on an induced $H_5$ labeled as in Figure \ref{fig:P7-star graphs}. 
            Let $C:=v_1-v_2-v_3-v_4-v_5-v_1$.  
            All indices below are modulo 5.
            
            Suppose first that $u$ is complete to $C$. 
            By Claim \ref{clm:a vertex mixed on a path via E} (2) with $P=v_j'-v_j-v_{j+1}-v_{j+2}$, we have $u$ is adjacent to $v_j'$ for each $j\in [5]$. 
            If $u$ is not adjacent to $w$, then $\{v_1',u,v_1,v_2,v_3,w\}$ induces a co-$E$, a contradiction. 
            So we may assume $uw\in E(G)$. 
            Then $u$ is complete to $H_5$, a contradiction. 
            So $u$ is not complete to $C$. 

            Now we suppose that $u$ is anti-complete to $C$. 
            By Claim \ref{clm:a vertex mixed on a path via E} (1) with $P=v_j'-v_j-v_{j+1}$, we have $u$ is not adjacent to $v_j'$ for each $j\in [5]$. 
            By Claim \ref{clm:a vertex mixed on a path via E} (1) with $P=v_1'-v_1-w$, we have $u$ is not adjacent to $w$. 
            It follows that $u$ is anti-complete to $H_5$, a contradiction. 
            So $u$ is not anti-complete to $C$. 

            Therefore, we may assume $u$ is mixed on $C$. 
            By Claim \ref{clm:a vertex mixed on a path via E} (1) and (2), there is an index $j\in [5]$ such that $N(u)\cap V(C)=\{v_j,v_{j-2},v_{j+2}\}$, say $\{v_1,v_{3},v_{4}\}$. 
            By Claim \ref{clm:a vertex mixed on a path via E} (1) with $P=v_4-w-v_2$, we have $u$ is adjacent to $w$. 
            By Claim \ref{clm:a vertex mixed on a path via E} (2) with $P=v_1'-v_1-w-v_3$, we have $u$ is adjacent to $v_1'$. 
            Then $\{u,w,v_1',v_1,v_2,v_3\}$ induces a co-$E$, a contradiction. This proves the first statement of the claim.
            
            The second statement of Claim \ref{clm:homo structure for E} follows immediately with $(x,y)$ replaced by $(v,a_i)$. 
            This completes the proof of Claim \ref{clm:homo structure for E}. 
        \end{subproof}
       
        Now we are ready to define $X_i,Y_i$ and $(A^i_1,\ldots, A^i_{t_i})$. 
        Let $X_i\subseteq B_i$ be the set of all vertices contained in some induced $H_5$ in $B_i$ and $Y_i=B_i\setminus X_i$. So $Y_i$ is $H_5$-free.
        Since $G$ is co-$E$-free, $Y_i$ is $\{H_5,\text{co-}E\}$-free.  
        This proves Lemma \ref{lem:coro-lem for E} (1). 
        For any two vertices $d,d'\in X_i$, we say $d,d'$ have {\em relation $\mathcal{R}$} if and only if there is a vertex sequence $d=d_1,d_2,\ldots,d_n=d'$ such that for each $k\in [n-1]$, $d_k$ and $d_{k+1}$ are contained in the same induced $H_5$ in $B_i$. 
        It is easy to check that $\mathcal{R}$ is an equivalence relation on $X_i$. 
        Let $\mathcal{L}^1$ be the blockade whose blocks are equivalence classes of $\mathcal{R}$. 
        Since $H_5$ is connected, each block of $\mathcal{L}^1$ is connected.
        For each $s\geq 2$, let $\mathcal{L}^{s}=\mathcal{L}^{s-1}/\mathcal{M}$. 
        Note that $\mathcal{L}^{s}$ is different from $\mathcal{L}^{s-1}$ if and only if there are two mixed blocks in $\mathcal{L}^{s-1}$.
        Since $X_i$ is finite, there is an integer $q$ such that $\mathcal{L}^{q}$ is a pure blockade. 
        Set $(A^i_1,\ldots, A^i_{t_i}):=\mathcal{L}^{q}$ and  
        this proves Lemma \ref{lem:coro-lem for E} (2.1). 
        If there is an induced $H_5$ such that each vertex of this $H_5$ is contained in a different block of $(A^i_1,\ldots,A^i_{t_i})$, then these vertices should have been in the same block of $\mathcal{L}^1$, a contradiction. Since $G$ is co-$E$-free, the patter graph is $\{H_5,\text{co-}E\}$-free.
        This proves Lemma \ref{lem:coro-lem for E} (2.2).
        It remains to prove Lemma \ref{lem:coro-lem for E} (2.3).
      
        \begin{claim}\label{clm:homo blocks for E}
            For every $s\in[q]$, every block $L$ of $\mathcal{L}^s$ and every vertex $u\in \bigcup_{k\in [\ell]\setminus \{i\}} B_k$, $u$ is pure to $L$.
        \end{claim}
        \begin{subproof}[Proof of Claim \ref{clm:homo blocks for E}]
            We prove this claim by induction on $s$. 
            If $s=1$, then this claim follows immediately from Claim \ref{clm:homo structure for E} with $(x,y,u)=(v,a_i,u)$. 
            So we may assume that $s\ge 2$ and this claim holds for $s-1$.

            Suppose to the contrary that there is a vertex $u\in \bigcup_{k\in [\ell]\setminus \{i\}} B_k$ that is mixed on $L$.  
            By the inductive hypothesis, $u$ is pure to each block of $\mathcal{L}^{s-1}$. 
            By Lemma \ref{clm:property of levels} (3), there are two mixed blocks of $\mathcal{L}^{s-1}$ contained in $L$ such that $u$ is complete to one of these two blocks but is anti-complete to the other one. 
            Suppose that for each fixed $i\in [s]$, there are two mixed blocks $D_1,D_2$ of $\mathcal{L}^i$ and three vertices $x,y,z$ with $z\in N(x)\setminus N(y)$ such that $x,y$ are complete to $D_1\cup D_2$ and $z$ is complete to $D_1$ and anti-complete to $D_2$. If $b_1\in D_1$ mixes on $D_2$, then $b_1$ mixed on an edge $b_2b'_2\in D_2$ by connectivity. 
            By Claim \ref{clm:a vertex mixed on a path via E} (1) with $(P,x,y,z)=(b_1-b_2-b_2',x,y,z)$, we obtain a contradiction since $z$ is mixed on $b_1-b_2-b_2'$ and has two consecutive non-neighbors $b_2,b_2'$.

            By applying Lemma \ref{clm:inductive} repeatedly, beginning with $(x,y,u)=(v,a_i,u)$, where each application of Lemma \ref{clm:inductive} preserves the same hypotheses until reaching $\mathcal{L}^1$, there are two mixed blocks $A_1,A_2$ of $\mathcal{L}^1$ and three vertices $x',y',u'\notin A_1\cup A_2$ such that 
            \begin{itemize}
                \item $x',y'$ are two non-adjacent vertices that are complete to $A_1\cup A_2$; 
                \item $u'\in N(x')\setminus N(y')$ is complete to $A_1$ but is anti-complete to $A_2$.  
            \end{itemize}  
            So we may assume that there is a vertex $u_2\in A_2$ mixed on $A_1$. 
            By Claim \ref{clm:homo structure for E} with $(x,y,u)=(y',u',u_2)$, $u_2$ is pure to every induced $H_5$ of $A_1$. 
            Since $A_1$ is an equivalence class of $\mathcal{R}$, $u_2$ is pure to $A_1$, a contradiction.  
            
            This complete the proof of Claim \ref{clm:homo blocks for E}. 
        \end{subproof} 
        
        By Claim \ref{clm:homo blocks for E} with $s=q$, Lemma \ref{lem:coro-lem for E} (2.3) holds.
        This completes the proof of Lemma \ref{lem:coro-lem for E}. 
    \end{proof}

    \subsection{Bird}

    In this subsection, we prove that Bird-free graph satisfies the hypothesis of  Lemma \ref{lem:property of co-E}.

    \begin{lemma}\label{lem:coro-lem}
        Let $G$ be a co-Bird-free graph with a $(\ell,w)$-comb $((a_i,B_i),i\in[\ell])$ in $G$ and a vertex $v\in V(G)\setminus(\{a_i:i\in[\ell]\}\cup\bigcup_{i\in[\ell]}B_i)$ such that $v$ is complete to $\bigcup_{i\in[\ell]}B_i$ and anti-complete to $\{a_i:i\in[\ell]\}$. 
        For each $i\in[\ell]$, $B_i$ can be partitioned into $X_i,Y_i$ such that 
        \begin{itemize}
            \item[$(1)$] $Y_i$ is $E$-graph-free;  
            \item[$(2)$] $X_i$ can be partitioned into $(A^i_1,\ldots, A^i_{t_i})$ such that 
            \begin{itemize}
                \item[$(2.1)$] $(A^i_1,\ldots, A^i_{t_i})$ is a pure blockade; 
                \item[$(2.2)$] the pattern of $(A^i_1,\ldots, A^i_{t_i})$, whose vertex set consists of all blocks of $(A^i_1,\ldots, A^i_{t_i})$ and two vertices are adjacent if and only if their corresponding blocks are complete to each other, is $E$-graph-free; 
                \item[$(2.3)$] for each $j\in [t_i]$ and every vertex $u\in \bigcup_{k\in [\ell]\setminus \{i\}} B_k$, $u$ is pure to $A^i_j$.
            \end{itemize}
        \end{itemize} 
    \end{lemma}
    \begin{proof}
        We begin with a useful claim. 
        \begin{claim}\label{clm:a vertex mixed on a path via bird}
            Let $S$ be an induced subgraph of $G$ and $x,y\in V(G)\setminus V(S)$ be two non-adjacent vertices complete to $S$. 
            For every induced subgraph $H$ of $S$, there is no $u\in N(x)\setminus N(y)$ for which either of the following holds. 
            \begin{itemize}
                \item[$(1)$] $H$ is $P_2+P_1$ and $u$ is mixed on the $P_2$ but is anti-complete to the $P_1$.  
                \item[$(2)$] $H$ is $P_3+P_1$ and $u$ is adjacent to the $P_1$ and two consecutive vertices of the $P_3$ but is anti-complete to the remaining vertex of the $P_3$.  
            \end{itemize} 
        \end{claim}
        
        \begin{subproof}[Proof of Claim \ref{clm:a vertex mixed on a path via bird}]
            Let $u\in N(x)\setminus N(y)$ be a vertex. 
            Suppose first that $H$ is $P_2+P_1$ with edge $ab$ and the isolated vertex $c$. 
            If $u$ is adjacent to $a$ but is non-adjacent to $b$ and $c$, then $\{x,y,u,a,b,c\}$ induces a co-Bird (see Figure \ref{fig:co-Bird in claim 5.2.1}), a contradiction. 
            This proves Claim \ref{clm:a vertex mixed on a path via bird} (1). 

            Now we suppose that $H$ is $P_3+P_1$ with induced path $a-b-c$ and the isolated vertex $d$. 
            If $u$ is adjacent to $b,c,d$ but is non-adjacent to $a$, then $\{y,u,a,b,c,d\}$ induces a co-Bird (see Figure \ref{fig:co-Bird in claim 5.2.1}), a contradiction. 
            This proves Claim \ref{clm:a vertex mixed on a path via bird} (2). 
        \end{subproof}
        
        \begin{figure}[h!]
        \centering
        \begin{tikzpicture}[scale=1]
        \tikzstyle{vertex}=[circle, draw, fill=white, inner sep=1pt, minimum size=5pt]
        
            \node[vertex](1) at (0,1){\scriptsize $a$};
            \node[vertex](2) at (0,0){\scriptsize $b$};
            \node[vertex](3) at (0,-1){\scriptsize $c$};
            
            \node[vertex](4) at (1,0){\scriptsize $x$};
            \node[vertex](5) at (-1,0){\scriptsize $y$};
            \node[vertex](6) at (1,1){\scriptsize $u$};

            \node at (0,-2.5) {Proof of Claim \ref{clm:a vertex mixed on a path via bird} (1)};
            
            \foreach \from/\to in {1/2}
            \draw (\from) -- (\to);

            \foreach \from/\to in {4/1,4/2,4/3}
            \draw (\from) -- (\to);

            \foreach \from/\to in {5/1,5/2,5/3}
            \draw (\from) -- (\to);

            \foreach \from/\to in {6/1,6/4}
            \draw (\from) -- (\to);

        \end{tikzpicture}
        \hspace{10mm}
        \begin{tikzpicture}[scale=1]
        \tikzstyle{vertex}=[circle, draw, fill=white, inner sep=1pt, minimum size=5pt]
        
            \node[vertex](1) at (0,1.5){\scriptsize $a$};
            \node[vertex](2) at (0,0.5){\scriptsize $b$};
            \node[vertex](3) at (0,-0.5){\scriptsize $c$};
            \node[vertex](4) at (0,-1.5){\scriptsize $d$};
            
            \node[vertex](5) at (-1,0){\scriptsize $y$};
            \node[vertex](6) at (1,1){\scriptsize $u$};

            \node at (0,-2.5) {Proof of Claim \ref{clm:a vertex mixed on a path via bird} (2)};
            
            \foreach \from/\to in {1/2,2/3}
            \draw (\from) -- (\to);

            \foreach \from/\to in {5/1,5/2,5/3,5/4}
            \draw (\from) -- (\to);

            \foreach \from/\to in {6/3,6/2,6/4}
            \draw (\from) -- (\to);

        \end{tikzpicture}
        \caption{Induced co-Birds in the proof of Claim \ref{clm:a vertex mixed on a path via bird}.}
        \label{fig:co-Bird in claim 5.2.1}
    \end{figure}
        
        \begin{claim}\label{clm:homo structure}
            If $x,y$ are two non-adjacent vertices complete to an induced $E$-graph,
            then every vertex $u\in N(x)\setminus N(y)$ is pure to this $E$-graph. 
            In particular, for every vertex $u\in \bigcup_{k\in [\ell]\setminus \{i\}} B_k$, $u$ is pure to every induced $E$-graph in $B_i$. 
        \end{claim}
        \begin{subproof}[Proof of Claim \ref{clm:homo structure}]
            Suppose to the contrary that there is a vertex $u\in N(x)\setminus N(y)$ that is mixed on an induced $E$-graph. 
            Let $P=v_1-v_2-v_3-v_4-v_5$ be the induced $P_5$ contained in this $E$-graph and $v_3'$ be the vertex adjacent to $v_3$. 
            
            Suppose first that $u$ is complete to $P$.  
            By Claim \ref{clm:a vertex mixed on a path via bird} (2) with $H$ induced by $\{v_3',v_3,v_4,v_1\}$, we have $u$ is adjacent to $v_3'$.   
            Then $u$ is complete to this $E$-graph, a contradiction. 
            So $u$ is not complete to $P$. 
            Now we suppose that $u$ is anti-complete to $P$. 
            By Claim \ref{clm:a vertex mixed on a path via bird} (1) with $H$ induced by $\{v_1,v_3,v_3'\}$, we have $u$ is not adjacent to $v_3'$. 
            It follows that $u$ is anti-complete to this $E$-graph, a contradiction. 
            So $u$ is not anti-complete to $P$. 

            Therefore, we may assume $u$ is mixed on $P$.
            Suppose first that $u$ is mixed on $v_1v_2$. 
            By Claim \ref{clm:a vertex mixed on a path via bird} (1), $u$ is adjacent to each of $v_3',v_4,v_5$.
            If $u$ is adjacent to $v_1$ but is not adjacent to $v_2$, then $u$ is adjacent to $v_3$
            by Claim \ref{clm:a vertex mixed on a path via bird} (2) with $H$ induced by $\{v_1,v_3,v_4,v_5\}$.
            Then the $\{v_2,v_3,v'_3,v_5\}$ contradicts 
            Claim \ref{clm:a vertex mixed on a path via bird} (2).
            So we may assume that $u$ is adjacent to $v_2$ but is not adjacent to $v_1$. 
            By Claim \ref{clm:a vertex mixed on a path via bird} (1) with $H$ induced by $\{v_1,v_3,v_3'\}$, $u$ is adjacent to $v_3$. Then $\{v_1,v_2,v_3,v_5\}$ contradicts Claim \ref{clm:a vertex mixed on a path via bird} (2).
            This proves that $u$ is pure to $v_1v_2$.
            By symmetry, $u$ is pure to $v_4v_5$. 

            If $u$ is adjacent to $v_1,v_2,v_4,v_5$, then $u$ is non-adjacent to $v_3$ and then $\{v_1,v_2,v_3,v_5\}$ contradicts Claim \ref{clm:a vertex mixed on a path via bird} (2).
            If $u$ is non-adjacent to $v_1,v_2,v_4,v_5$, then $u$ is adjacent to $v_3$ and so $\{v_2,v_3,v_5\}$ contradicts Claim \ref{clm:a vertex mixed on a path via bird} (1).
            So we may assume that $u$ is adjacent to $v_1,v_2$ but is non-adjacent to $v_4,v_5$. 
            By Claim \ref{clm:a vertex mixed on a path via bird} (1) with $H$ induced by $\{v_2,v_3,v_5\}$, $u$ is adjacent to $v_3$. 
            By Claim \ref{clm:a vertex mixed on a path via bird} (2) with $H$ induced by $\{v_1,v_3',v_3,v_4\}$, $u$ is not adjacent to $v_3'$, which gives a contradiction by Claim \ref{clm:a vertex mixed on a path via bird} (1) with $H$ induced by $\{v_3,v_3',v_5\}$. 
            This proves the first statement of the claim.
            
            The second statement of Claim \ref{clm:homo structure} follows immediately with $(x,y)$ replaced by $(v,a_i)$. 
            This completes the proof of Claim \ref{clm:homo structure}. 
        \end{subproof}
        
        Now we are ready to define $X_i,Y_i$ and $(A^i_1,\ldots, A^i_{t_i})$. 
        Let $X_i\subseteq B_i$ be the set of all vertices contained in some induced $E$-graph in $B_i$ and $Y_i=B_i\setminus X_i$. So $Y_i$ is $E$-graph-free.   
        This proves Lemma \ref{lem:coro-lem} (1). 
        For any two vertices $d,d'\in X_i$, we say $d,d'$ have {\em relation $\mathcal{R}$} if and only if there is a vertex sequence $d=d_1,d_2,\ldots,d_n=d'$ such that for each $k\in [n-1]$, $d_k$ and $d_{k+1}$ are contained in the same induced $E$-graph in $B_i$. 
        It is easy to check that $\mathcal{R}$ is an equivalence relation on $X_i$. 
        Let $\mathcal{L}^1$ be the blockade whose blocks are equivalence classes of $\mathcal{R}$. 
        Since $E$-graph is anti-connected, each block of $\mathcal{L}^1$ is anti-connected. 
        For each $s\geq 2$, let $\mathcal{L}^{s}=\mathcal{L}^{s-1}/\mathcal{M}$. 
        Note that $\mathcal{L}^{s}$ is different from $\mathcal{L}^{s-1}$ if and only if there are two mixed blocks in $\mathcal{L}^{s-1}$.
        Since $X_i$ is finite, there is an integer $q$ such that $\mathcal{L}^{q}$ is a pure blockade. 
        Set $(A^i_1,\ldots, A^i_{t_i}):=\mathcal{L}^{q}$ and  
        this proves Lemma \ref{lem:coro-lem} (2.1). 
        If there is an induced $E$-graph such that each vertex of this $E$-graph is contained in a different block of $(A^i_1,\ldots,A^i_{t_i})$, then these vertices should have been in the same block of $\mathcal{L}^1$, a contradiction. 
        This proves Lemma \ref{lem:coro-lem} (2.2).
        It remains to prove Lemma \ref{lem:coro-lem} (2.3).
      
        \begin{claim}\label{clm:homo blocks}
            For every $s\in[q]$, every block $L$ of $\mathcal{L}^s$ and every vertex $u\in \bigcup_{k\in [\ell]\setminus \{i\}} B_k$, $u$ is pure to $L$.
        \end{claim}
        \begin{subproof}[Proof of Claim \ref{clm:homo blocks}]
            We prove this claim by induction on $s$. 
            If $s=1$, then this claim follows immediately from Claim \ref{clm:homo structure} with $(x,y,u)=(v,a_i,u)$. 
            So we may assume that $s\ge 2$ and this claim holds for $s-1$.

            Suppose to the contrary that there is a vertex $u\in \bigcup_{k\in [\ell]\setminus \{i\}} B_k$ that is mixed on $L$.  
            By the inductive hypothesis, $u$ is pure to each block of $\mathcal{L}^{s-1}$. 
            By Lemma \ref{clm:property of levels} (3), there are two mixed blocks of $\mathcal{L}^{s-1}$ contained in $L$ such that $u$ is complete to one of these two blocks but is anti-complete to the other one. 
            Suppose that for each fixed $i\in [s]$, there are two mixed blocks $D_1,D_2$ of $\mathcal{L}^i$ and three vertices $x,y,z$ with $z\in N(x)\setminus N(y)$ such that $x,y$ are complete to $D_1\cup D_2$ and $z$ is complete to $D_1$ and anti-complete to $D_2$. If $b_1\in D_1$ mixes on $D_2$, then $b_1$ mixed on a non-edge $b_2b'_2\in D_2$ by anti-connectivity. 
            By Claim \ref{clm:a vertex mixed on a path via bird} (1) with $(H,x,y,z)=(\{b_1b_2\}+b_2',x,y,z)$, we obtain a contradiction.

            By applying Lemma \ref{clm:inductive} repeatedly, beginning with $(x,y,u)=(v,a_i,u)$, where each application of Lemma \ref{clm:inductive} preserves the same hypotheses until reaching $\mathcal{L}^1$, there are two mixed blocks $A_1,A_2$ of $\mathcal{L}^1$ and three vertices $x',y',u'\notin A_1\cup A_2$ such that 
            \begin{itemize}
                \item $x',y'$ are two non-adjacent vertices that are complete to $A_1\cup A_2$, and
                \item $u'\in N(x')\setminus N(y')$ is complete to $A_1$ but is anti-complete to $A_2$.  
            \end{itemize}
            
            So we may assume that there is a vertex $u_2\in A_2$ mixed on $A_1$. 
            By Claim \ref{clm:homo structure} with $(x,y,u)=(y',u',u_2)$, $u_2$ is pure to every induced $E$-graph of $A_1$. 
            Since $A_1$ is an equivalence class of $\mathcal{R}$, $u_2$ is pure to $A_1$, a contradiction.  
            
            This complete the proof of Claim \ref{clm:homo blocks}. 
        \end{subproof} 

        By Claim \ref{clm:homo blocks} with $s=q$, Lemma \ref{lem:coro-lem} (2.3) holds.
        This completes the proof of Lemma \ref{lem:coro-lem}. 
    \end{proof}
    
\section{Concluding remarks}

    In this paper, we combine several new ideas to extend the iterative sparsification framework recently developed by Nguyen, Scott and Seymour \cite{NSS23}. Our framework not only yields the first two six-vertex graphs satisfying the Erdős-Hajnal conjecture which do not follow directly from Theorem \ref{thm:E-H-substitute} and Corollary \ref{coro:degree 1 and degree n-2},
    but also provides a general criterion (Lemma \ref{lem:proving (*)}) that may be used to verify the conjecture for further graph classes. Moreover, we recover the result of Chudnovsky and Safra \cite{CS08} on bull graph and the result of Nguyen, Scott and Seymour \cite{NSS23} on $P_5$ in a unified way.  
    This suggests that the comb-based approach is a promising direction for future work on the Erdős-Hajnal conjecture.

\section*{Acknowledgments}
We thank Viet Hoang and his coauthors for pointing out a gap in an earlier version of our proof of Lemma \ref{lem:get a blockade by iteration}, and for their helpful comments.

\end{document}